\documentclass[12pt]{amsart}
\usepackage{sfilip_presets,amsmath,amsthm}
\usepackage[stretch=10, final]{microtype}
\usepackage{lmodern}
\usepackage[T1]{fontenc}
\usepackage{titletoc}

\let\oldtocsection=\tocsection
\let\oldtocsubsection=\tocsubsection
\let\oldtocsubsubsection=\tocsubsubsection

\renewcommand{\tocsection}[2]{\hspace{0em}\oldtocsection{#1}{#2}}
\renewcommand{\tocsubsection}[2]{\hspace{1.75em}\oldtocsubsection{#1}{#2}}
\renewcommand{\tocsubsubsection}[2]{\hspace{2em}\oldtocsubsubsection{#1}{#2}}
\makeatletter
\renewcommand{\paragraph}{%
\@startsection {paragraph}{4}
{\z@} \z@ {-\fontdimen 2\font }\bfseries
}
\makeatother

\newif \iffig
\figfalse

\newif\ifdebug
\debugfalse

\ifdebug
  \usepackage[status=draft,author=]{fixme}
  \fxsetup{inline, marginclue,theme=color}
\else
  \usepackage[status=final, author="Simion Filip"]{fixme}
\fi
\usepackage{hyperref}
\usepackage{color}
\definecolor{darkred}{rgb}{0.4,0,0}
\definecolor{darkgreen}{rgb}{0,0.5,0}
\definecolor{darkblue}{rgb}{0,0,0.4}

\hypersetup{
    pdftitle={Zero exponents and monodromy of the KZ cocycle},	% title
    pdfauthor={Simion Filip},	% author
    pdfsubject={dynamical systems},		% subject of the document
    pdfkeywords={},	% list of keywords
    pdfnewwindow=true,		% links in new window
    colorlinks=true,		% false: boxed links; true: colored links
    linkcolor=darkblue,		% color of internal links (change box color with linkbordercolor)
    citecolor=darkred,		% color of links to bibliography
    filecolor=darkblue,		% color of file links
    urlcolor=darkblue,		% color of external links
    pdfborder={0 0 0},
    breaklinks=true
}
\def\subsek~{\S{}}

\numberwithin{equation}{section}

\newtheoremstyle{mytheoremstyle} % name
    {5pt}                    % Space above
    {5pt}                    % Space below
    {\itshape}                   % Body font
    {}                           % Indent amount
    {\bfseries}                   % Theorem head font
    {}                          % Punctuation after theorem head
    {.5em}                       % Space after theorem head
    {}  % Theorem head spec (can be left empty, meaning ‘normal’)

\usepackage{aliascnt}
    
\theoremstyle{mytheoremstyle}
    
\newtheorem{theorem}{Theorem}[section]

\newaliascnt{lemma}{theorem}  
\newtheorem{lemma}[lemma]{Lemma}  
\aliascntresetthe{lemma}

\newaliascnt{proposition}{theorem}
\newtheorem{proposition}[proposition]{Proposition}
\aliascntresetthe{proposition}

\newaliascnt{corollary}{theorem}  
\newtheorem{corollary}[corollary]{Corollary}  
\aliascntresetthe{corollary}

\newaliascnt{exercise}{theorem}  
  
\aliascntresetthe{exercise}

\newaliascnt{definition}{theorem}  
  
\aliascntresetthe{definition}

\newaliascnt{remark}{theorem}  
\newtheorem{remark}[remark]{Remark}  
\aliascntresetthe{remark}

\newaliascnt{example}{theorem}  
\newtheorem{example}[example]{Example}  
\aliascntresetthe{example}

\newaliascnt{question}{theorem}  
\newtheorem{question}[question]{Question}  
\aliascntresetthe{question}

\def\equationautorefname#1#2\null{%
  Eq.#1(#2\null)%
}

\usepackage{etoolbox}
\makeatletter
\patchcmd{\@maketitle}
  {\ifx\@empty\@dedicatory}
  {\ifx\@empty\@date \else {\vskip3ex \centering\footnotesize\@date\par\vskip1ex}\fi
   \ifx\@empty\@dedicatory}
  {}{}
\patchcmd{\@adminfootnotes}
  {\ifx\@empty\@date\else \@footnotetext{\@setdate}\fi}
  {}{}{}
\makeatother

\begin{document}

\title[Zero exponents and monodromy of the KZ cocycle]{Zero Lyapunov exponents and monodromy of the Kontsevich-Zorich cocycle}

%
%	change to \title[short title]{Actual title}
%	if what appears on headers is too long

%		Date			
%------------------------------------
\thanks{Revised \textsc{\today}}

%-----------Date First posted---------------------
\date{October 2014}

%-----------Author info---------------------
\author{\vspace{-0.5em}Simion Filip\vspace{-0.5em}}

%		Address
 \address{
 \parbox{0.5\textwidth}{
 Department of Mathematics\\
 University of Chicago\\
 Chicago IL, 60615\\}
 	}
 \email{{sfilip@math.uchicago.edu}}

\begin{abstract}
We describe all the situations in which the Kontse\-vich-Zorich cocycle has zero Lyapunov exponents.
Confirming a conjecture of Forni, Matheus, and Zorich, this only occurs when the cocycle satisfies additional geometric constraints.
We also describe the real Lie groups which can appear in the monodromy of the Kontsevich-Zorich cocycle.
The number of zero exponents is then as small as possible, given its monodromy.
\vspace{-1em}
\end{abstract}

\maketitle

%		TOC
%  \noindent\hrulefill
\vspace{-1em}
 \tableofcontents
%  \nointerlineskip
%  \noindent \hrulefill
%====================================================

%		Things to be fixed
\ifdebug
  \listoffixmes
\fi
%====================================================

\section{Introduction}

Consider a genus $g$ surface equipped with a foliation (necessarily with singular points if $g>1$).
Take a typical leaf of the foliation of some large length $T$.
Closing it up and considering its homology class, it grows linearly in $T$.
In the 90s, Zorich (see \cite{Zorich_leaves}) discovered that exact lower-order asymptotics exist, with $g$ terms of orders $T^{\lambda_i}$ for some $1=\lambda_1\geq\cdots \geq \lambda_g$.
The numbers $\{\lambda_i\}$ are, in fact, the Lyapunov exponents of a different dynamical system, on the moduli space of Riemann surfaces.
They measure the growth rate of sections of a vector bundle, transported along orbits of the \Teichmuller geodesic~flow.

\paragraph{Flat surfaces}
Let $(X,\omega)$ be a pair consisting of a Riemann surface with a holomophic $1$-form.
A moduli space of all such objects of fixed topological type is called a stratum, denoted $\cH(\kappa)$, where $\kappa$ encodes the multiplicities of zeros.
The surveys of Masur-Tabachnikov \cite{MT}, Forni-Matheus \cite{ForniMatheus_Lectures}, or Zorich \cite{Zorich}, can serve as an introduction to the subject.

The group $\SL_2\bR$ acts on this space and preserves a natural measure of Lebesgue class.
It is finite by work of Masur and Veech \cite{Masur,Veech}.

By recent work of Eskin and Mirzakhani \cite{EM} all other ergodic invariant measures are of Lebesgue class, supported on manifolds (in fact, algebraic varieties \cite{sfilip_algebraicity}).
Further results concerning orbit closures and equidistribution, in analogy with Ratner's theorems, are developed by Eskin, Mirzakhani and Mohammadi \cite{EMM}.

\paragraph{Lyapunov exponents}
For each invariant measure, one considers the Kontsevich-Zorich cocycle and its Lyapunov exponents (see \cite{Forni_spectrum}).
Their study was initiated by the work of Zorich \cite{Zorich_gauss,Zorich_leaves}, followed by a formula for their sum by Kontsevich \cite{Kontsevich} (see also Eskin-Kontsevich-Zorich \cite{EKZ_sum}).

Chaika and Eskin \cite{Chaika_Eskin} show that it suffices to consider individual $\SL_2\bR$ orbits for the asymptotics in the Oseledets theorem to hold.
The behavior then is dictated by the orbit closure.

For applications (e.g. to the wind-tree model \cite{windtree}) it is useful to know when the Lyapunov spectrum has degeneracies, e.g. multiplicities or zero exponents.
Confirming a conjecture of Zorich, for strata Avila and Viana \cite{AvilaViana} showed that the spectrum is simple.

The question of zero exponents was investigated from several points of view.
The situation when the cocycle has an $\SL_2\bR$-invariant isometric piece was considered by Aulicino \cite{Aulicino_disks, Aulicino_affine} as well as M\"oller \cite{Moller_Shimura}.
A geometric criterion of Forni \cite{Forni_geometric} gives an upper bound for the number of zero exponents.

\subsection{The main result}
A general mechanism for zero exponents was described in some examples by Forni, Matheus, and Zorich \cite{FMZ_zero, FMZ_bundles}.
It was based on an abundance of examples, as constructed by Forni-Matheus-Zorich \cite{FMZ_square}, Eskin-Kontsevich-Zorich \cite{EKZ_square}, and McMullen \cite{McMullen_Hodge}.
They conjectured in \cite{FMZ_zero} that this is the only situation in which zero exponents occur.

The purpose of this paper is to prove their conjecture in a slightly refined form, to account for other possibilities.
The main result (see \autoref{thm:KZ_zero}) is as follows.

\begin{theorem}
\label{thm:main_zero_exp}
 Let $\cM$ be an affine invariant manifold of a stratum of translation surfaces, equipped with an ergodic $\SL_2(\bR)$-invariant probability measure.
 For the Kontsevich-Zorich cocycle over $\cM$, the number of zero exponents is precisely equal to the constraints predicted by the monodromy.
 
 Concretely, let $E_\bR$ be one of the flat $\bR$-irreducible pieces of the KZ cocycle.
 Let $G$ be the Zariski closure of the monodromy of $E_\bR$.
 
 Then zero exponents in $E_\bR$ can occur if and only if we are in the following situation.
 The group $G$ has at most one non-compact factor, equal up to finite index to $\SU_{p,q}$, for some $p> q$, or $\SO_{2n}^*$ and $n$ is odd.
 
 The representation in which $\SU_{p,q}$ occurs is the standard one, or an exterior power of the standard.
 In the standard representation, there are $2(p-q)$ zero exponents.
 
 If it is $\SU_{p,q}$ in the $k$-th exterior power of the standard with $k\geq 2$, then necessarily $q=1$.
 The number of zero exponents is then $\binom{p-1}{k-2} + \binom{p-1}{k} $.
 This number is minimum possible, given the monodromy constraint (see \autoref{sec:geom_oseledets} and \autoref{cor:zero_is_rk} for details).
 
 If the group is $\SO^*_{2n}$, then zero exponents occur only if $n$ is odd, in which case there are precisely four.
 
 Moreover, the number of strictly positive exponents bounds above the rank of the second fundamental form, cf. Problem 1 asked by Forni, Matheus and Zorich in \cite{FMZ_bundles}.
\end{theorem}
We also give a classification of possible groups in the ``locally flat'' algebraic hull (see \autoref{thm:gps_alg_hull} and \autoref{sec:monodromy}).
\begin{theorem}
\label{thm:monodromy_classification}
 From \cite{sfilip_ssimpleKZ} the Kontsevich-Zorich cocycle is semisimple, and its decomposition respects the Hodge structure.
 Consider an $\bR$-irreducible piece, and let $G$ be the corresponding semisimple group in the algebraic hull.
 
 Then $G$ has at most one non-compact factor, and it lies in a certain representation.
 At the level of Lie algebras, the corresponding real Lie algebra and representation must be one from the list
\begin{enumerate}
 \item [(i)] $\su_{p,q}$ in the standard representation, or $\su_{p,1}$ in any exterior power representation.
 \item [(ii)] $\so_{2n-1,2}(\bR)$ in the spin representation.
 \item [(iii)] $\fraksp_{2g}(\bR)$ in the standard representation.
 \nopagebreak[1]
 \item [(iv)] $\so_{2n}^*$ in the standard representation, or $\so_{2n-2,2}(\bR)$ in either of the spin representations.
\end{enumerate}
 The classification applies to both the $\SL_2\bR$-invariant and flat semisimple decompositions (see \autoref{thm:SL_2_invar}).
\end{theorem}

After a preliminary version of this paper was circulated, in joint work with Matheus and Forni \cite{FFM_Quaternions} we found instances with monodromy in the group $\SO_{6}^*$ in its standard representation.
This coincides with $\SU_{3,1}$ in the second exterior power representation.
It would be curious to find orthogonal groups in spin representations occurring in the monodromy of the Kontsevich-Zorich cocycle.

\subsection{Some consequences}

The above general results on zero exponents have particular instances which could be of interest.
Some of them were previously obtained by Forni \cite{Forni_deviations,Forni_geometric,Forni_spectrum}.

\begin{corollary}
 Let $\cM$ be an affine invariant manifold of a stratum of flat surfaces.
 In particular, $\cM$ could be the whole stratum or a \Teichmuller curve.
 
 \begin{enumerate}
  \item[(i)] Let $p(T\cM)$ be the subbundle of the Kontsevich-Zorich cocycle corresponding to its tangent space.
  Then no zero exponents can occur in $p(T\cM)$, or any of its Galois conjugates.
  \item[(ii)] The Zariski closure of the monodromy in $p(T\cM)$ or any of its Galois conjugates is the corresponding full symplectic group.
 \end{enumerate}
\end{corollary}
\begin{proof}
 By \autoref{thm:main_zero_exp} which classifies the situations with zero exponents, the claim about Zariski closures of the monodromy implies the one about zero exponents.
 Moreover, proving the claim about Zariski closures for the piece of the cocycle corresponding to $p(T\cM)$ implies it for the other pieces, since the Zariski closures have to be isomorphic after extending scalars to $\bC$.
 Indeed, in the list from \autoref{thm:monodromy_classification}, only one real form of the symplectic group occurs, so this is the only possibility.
 
 To show that on $p(T\cM)$ the monodromy has Zariski closure the full symplectic group, there are two options.
 The first one is to note that the top Lyapunov exponent has multiplicity one, and the only item in \autoref{thm:monodromy_classification} which allows this is the symplectic group.
 
 The other option is to invoke the closing lemma, as used for instance by Wright \cite{Wright_field}.
 This implies the representation of the monodromy group has a dense collection of simple highest weight vectors (i.e. eigenvectors with highest eigenvalue) for diagonalizable elements of the monodromy.
 Again, looking at the list from \autoref{thm:monodromy_classification} yields the claim.
\end{proof}

\subsection{Outline of the paper}

\autoref{sec:refinement} contains the heart of the argument.
In known examples the Oseledets filtration has a further refinement over $\bC$.
\autoref{prop:g_t-invar} proves this refinement must exists whenever zero exponents occur.
This is the key step, which is then combined with dynamical arguments to get restrictions on the monodromy.

\autoref{sec:monodromy} extracts consequences for the monodromy.
\autoref{prop:real_rk} shows that the real rank of the group (rather, representation) is at most the number of non-zero exponents.
Real rank measures, informally, the number of ``interesting'' eigenvalues (e.g. not on the unit circle).
Therefore, this argument provides the desired upper bound on the number of zero exponents.
The second part of \autoref{sec:monodromy} analyzes the restrictions on monodromy coming from Hodge theory.

\autoref{sec:overview_Oseledets} contains a geometric point of view on the Oseledets theorem.
It connects rates of diffusion in symmetric spaces with Lyapunov exponents.
It also provides a unified way to describe Lyapunov exponents depending on the representation.
The real groups and representations which can occur in the Kontsevich-Zorich cocycle are made explicit.

\autoref{sec:zero_exponents} combines the results and deduces the For\-ni-Ma\-the\-us-Zo\-rich conjecture.
A construction in \autoref{subsec:caut_example} explains why in the exterior power cases, the number of zero exponents is not predicted by the signature of the indefinite metric.

\paragraph{Some general remarks}
Throughout, the groups $\SU_{p,q}$ appear in various representations, e.g. on $\bC^{p+q}$.
These are \emph{real} Lie groups, but act on vector spaces which also have a complex structure.
We view the representations as \emph{real} vector spaces, e.g. $\bC^{p+q}\cong \bR^{2(p+q)}$.
However, we also keep track of the action of the algebra $\bC$ on the representation.
Since it commutes with all other structures, it will act, for example, on Oseledets spaces.

As a typical consequence, the Oseledets spaces of the real vector bundles will be real even-dimensional, and will carry an action of $\bC$.

Bundle are complexified only when considering the subbundles from the variation of Hodge structure.
Expressions such as $\bC^n\otimes_\bR \bC$ will be avoided, although they implicitly appear in the arguments.

Moreover, we always work in some finite cover of a stratum where orbifold issues do not appear.
Therefore, the Kontsevich-Zorich cocycle is an honest cocycle, and the period coordinates are well-defined.

\paragraph{Acknowledgements}
I am grateful to Giovanni Forni for asking me the question about zero exponents and for discussions on this topic.
He also provided very useful feedback on a preliminary version of this paper.
I am also grateful to Madhav Nori for some useful remarks about semisimple Lie groups and Hodge structures, in particular about~$\SO^*_{2n}$.

I also had useful discussions with Carlos Matheus, Curtis McMullen, Alex Wright, and Anton Zorich.
I am also grateful to my advisor, Alex Eskin, for discussions on this topic.

\section{Refining the Oseledets filtration}
\label{sec:refinement}

This section contains the main dynamical part of the argument.
We first present an example that motivates the subsequent constructions.
Next, \emph{assuming the cocycle has zero exponents}, we refine the Oseledets filtration.
This is accomplished using Forni's formula for the partial sum of exponents.

The first step refines the filtration along the orbits of the \Teichmuller geodesic flow.
An argument from partially hyperbolic dynamics gives invariance along unstable leaves.
Combined with the holomorphic dependence of the space of $(1,0)$-forms, this puts restrictions on the monodromy.

\subsection{Motivation}

This section describes a refinement of the Oseledets filtration when the Kontsevich-Zorich cocycle has $\SU_{p,q}$ components.
Later we show that when zero exponents occur, this refined filtration must exist.
This structure will be used to analyze the algebraic hull.

To begin, suppose $E$ is an irreducible piece over $\bR$ of the semisimple decomposition of the local system underlying the KZ cocycle.
Suppose its monodromy is contained in $\SU_{p,q}$.
A large set of examples was described by McMullen, Forni-Matheus-Zorich, Matheus-Yoccoz-Zmiaikou \cite{McMullen_Hodge, FMZ_bundles,FMZ_zero, MYZ_origami}.
In those examples, the complexified bundle $E_\bC$ has a further flat splitting, corresponding to eigenspaces of some symmetry of all flat surfaces in that family.

This happens more generally when the monodromy is contained in $\SU_{p,q}$ (acting in the standard representation).
In such a case, the monodromy action commutes with the scaling by $\bC^\times$.
Recall $\SU_{p,q}$ is a \emph{real} Lie group acting on $\bC^{p+q}=\bR^{2(p+q)}$.

This implies we have a decomposition of the complexified bundle $E_\bC=E_+\oplus E_-$.
The $E_+$ bundle corresponds to vectors on which $z\in \bC^\times$ acts by $z$, while $E_-$ to those on which $z\in \bC^\times$ acts by $\conj{z}$.
Complex conjugation in $E_\bC$ exchanges $E_+$ and $E_-$.

We also have a decomposition coming from the Hodge structure
\begin{align*}
 E_+&= E^{1,0}_+\oplus E^{0,1}_+\\
 E_-&= E^{1,0}_-\oplus E^{0,1}_-
\end{align*}
Under complex conjugation $E^{a,b}_+=\conj{E^{b,a}_-}$.
The dimensions are $\dim E^{1,0}_+=p$, $\dim E^{0,1}_+=q$ and $p\geq q$.

Next, consider the Oseledets decomposition $E=E^{>0}\oplus E^0 \oplus E^{<0}$ and denote $E^{\geq 0}:=E^0\oplus E^{>0}$.
This can be further refined to the bundles $E_\pm$, and we only care about the sign of the exponents.

On each piece we have that $\dim E^{>0}_\pm=\dim E^{<0}_\pm=q$ and $\dim E^0_\pm=p-q$.
Therefore $\dim E^{\geq 0}_\pm=p$.
Note however that $\dim E^{1,0}_+=p$, i.e. its codimension inside $E_+$ is $q$.
Define the intersection
$$
E_{F^1}^{\geq 0} := E^{1,0}_+\cap E^{\geq 0}_+
$$
It has\footnote{Apriori, at least $p-q$, but isotropy conditions impose equality (see next section)} dimension $p-q$ and moreover it can be defined alternatively as the intersection $E^{1,0}\cap E^{\geq 0}$.
Indeed, on the $E_-$ component, the spaces don't intersect.

The key observation is that although $E^{\geq 0}_{F^1}$ is not $g_t$-invariant, the direct sum $E^{\geq 0}_{F^1}\oplus E^0$ is, in fact, invariant under $g_t$.
Since $E^{\geq 0}_{F^1}$ provides a complement to $E^{>0}_+$ inside $E^{\geq 0}_+$ we have the equality
$$
E^{\geq 0}_{F^1}\oplus E^{>0} = E^{0}_+ \oplus E^{>0}_+ \oplus E^{>0}_-
$$
The right-hand side is manifestly $g_t$-invariant, while the left-hand is not obviously so.
We shall prove that this phenomenon occurs whenever the bundle has zero exponents.

\subsection{Setup}

We have an $\SL_2\bR$-invariant probability measure, of Le\-besgue class on an affine manifold $\cM$.
Denote by $g_t$ the \Teichmuller geodesic flow.
Consider an $\SL_2\bR$-invariant subbundle $E$, defined over $\bR$ and of dimension $2g$.
The Oseledets theorem gives a $g_t$-invariant decomposition according to the sign of exponents
$$
E = E^{<0}\oplus E^0 \oplus E^{>0}
$$
Introduce the further notation $E^{\geq 0}:=E^0\oplus E^{>0}$.
Assume that we have zero exponents, i.e. $E^0$ is non-trivial.
Recall the basic properties of the decomposition:
\begin{enumerate}
 \item[(i)] $E^{>0}$ is an isotropic subspace
 \item[(ii)] The symplectic-orthogonal of $E^{>0}$ is $E^{\geq 0}$.
 \item[(iii)] The symplectic form on $E^0$ is non-degenerate, thus $\dim_\bR E^0=2k$.
\end{enumerate}
Introduce the notation $F^1:=H^{1,0}$ and define subspaces
\begin{align*}
 E^{\geq 0}_{F^1}&:= F^1\cap E^{\geq 0}_\bC\\
 E^{\geq 0}_{\conj{F^1}}&:= \conj{F^1}\cap E^{\geq 0}_\bC 
\end{align*}
These spaces are complex-conjugates of each other, because $E^{\geq 0}$ is defined over $\bR$.

\begin{lemma}
\label{lemma:dir_sum_orthog}
 We have that $\dim_\bC E^{\geq 0}_{F^1}=k$.
 In fact, the following decomposition holds
 $$
 E^{\geq 0}_\bC= E^{\geq 0}_{\conj{F^1}} \oplus E^{\geq 0}_{F^1}\oplus E^{>0}_\bC
 $$
 Moreover, the decomposition is Hodge-orthogonal.
\end{lemma}
\begin{proof}
 The intersection of $F^1$ and $E^{\geq 0}_\bC$ has dimension at least $k$.
 Indeed, the first space has codimension $g$, the second codimension $g-k$.
 
 Next, we claim $F^1\cap E^{>0}_\bC=\{0\}$.
 If $\alpha$ is in the intersection we also know $\conj{\alpha}\in E^{>0}_\bC$ since this space is defined over $\bR$.
 Because $E^{>0}$ is isotropic, the symplectic pairing of $\alpha$ and $\conj{\alpha}$ must vanish.
 This is a contradiction since $\alpha$ is holomorphic.
 
 Now observe that the above two properties hold for $F^1$ replaced with $\conj{F^1}$.
 Moreover $F^1$ and $\conj{F^1}$ don't intersect.
 This yields the direct sum decomposition, with summands of dimensions $k,k$, and $g-k$.
 
 Let us now prove Hodge-orthogonality.
 Recall $E^{\geq 0}$ is the same as the symplectic orthogonal of $E^{>0}$.
 Take a real class $c\in E^{>0}$ and decompose it according to $(1,0)$ and $(0,1)$ types
 $$
 c=\alpha\oplus \conj{\alpha}
 $$
 Take $\beta \in E^{\geq 0}\cap F^1$, then $(c,\beta)=0$ where $(,)$ is the symplectic pairing.
 Therefore $\alpha$ and $\beta$ are Hodge-orthogonal, for all $\alpha$ coming from a real class $c\in E^{>0}$.
 Therefore, the Hodge inner product of $\beta$ and $c$ also vanishes, for all real classes $c\in E^{>0}$.
 But these span $E^{>0}$, so $E^{\geq 0}_{F^1}$ is orthogonal to it.
 
 Similarly, the same calculation shows that $\conj{\beta}$ is Hodge-orthogonal to any class in $E^{>0}$.
 The complex conjugate space $E^{\geq 0}_{\conj{F^1}}$ satisfies the same property.
 Finally, any $(1,0)$ and $(0,1)$ subspaces are automatically Hodge-orthogonal, so $E^{\geq 0}_{F^1}$ and $E^{\geq 0}_{\conj{F^1}}$ are Hodge-orthogonal.
\end{proof}

Note that the decomposition from \autoref{lemma:dir_sum_orthog} is \emph{not} invariant under $g_t$.
We haven't yet used anything about zero exponents.
However, we have the following claim.

\begin{proposition}[Key Proposition]
\label{prop:g_t-invar}
 The following subbundle is $g_t$-invariant
 $$
 E^{\geq 0}_{F^1}\oplus E^{>0}
 $$
 In other words, along the $g_t$-flow, the bundle $E^{\geq 0}_{F^1}$ can only move in the $E^{>0}$ direction.
\end{proposition}

\begin{remark}
 A similar statement holds for the complex-conjugate bundle $E^{\geq 0}_{\conj{F^1}}$.
 One can also consider a time reversal and obtain a statement for $E^{\leq 0}:=E^{<0}\oplus E^0$.
\end{remark}

\subsection{Some preliminaries for \autoref{prop:g_t-invar}}

Recall the Forni formula (see \cite[Theorem 1]{FMZ_bundles} and \cite{Forni_deviations}) for the partial sum of exponents
$$
\lambda_1+\cdots+ \lambda_{g-k} = \int \Phi_{g-k}(\omega, E^{>0}(\omega))d\mu(\omega)
$$
The function $\Phi_{g-k}$ is defined by (see \cite[Lemma 2.8]{FMZ_bundles})
$$
\Phi_{g-k}(\omega, E^{0}) = \Lambda_1(\omega)+ \cdots \Lambda_g(\omega) - \sum_{i,j=g-k+1}^g
|B_\omega^\bR(c_i,c_j|^2
$$
The terms $\Lambda_i$ correspond to the singular values of the inner product defined by the second fundamental form.
The pairing $B_\omega$ is defined by 
$$
B_\omega(c_i,c_j)=\ip{A_\omega h(c_i)}{h(c_j)}
$$
where $A_\omega$ is the second fundamental form, and $h(-)$ denotes the holomorphic $1$-form with given real part.
Finally, the $c_i$ are defined as a Hodge-orthonormal basis of a (real) Lagrangian space, such that the first $g-k$ give a basis of the space $E^{>0}$.

We now re-express the above formula in terms of the spaces $E^{1,0}$ and $E^{0,1}$ of holomorphic and anti-holomorphic $1$-forms.
We view the second fundamental form as a map $\sigma: E^{1,0}\to E^{0,1}$.

Consider the space $E^{>0}$, complexify it and project to $E^{0,1}$ to obtain the space $E^{>0}_{0,1}$.
It is of complex dimension $g-k$.
We also have its Hodge-orthgonal $(E^{>0}_{0,1})^\perp$ inside $E^{0,1}$, and we can also take complex-conjugates.
This yields the decomposition, where complex conjugation swaps $(1,0)$ and $(0,1)$
\begin{equation}
\label{eqn:E_1001perp}
  \begin{aligned}
  E^{0,1} &= E^{>0}_{0,1}\oplus (E^{>0}_{0,1})^\perp \\
  E^{1,0} &= E^{>0}_{1,0}\oplus (E^{>0}_{1,0})^\perp
  \end{aligned}
\end{equation}
Viewing the second fundamental form as a map $\sigma:E^{1,0}\to E^{0,1}$, we can take its components for the above decomposition.
We are interested in the lower-right corner
$$
\sigma_{22}: (E^{>0}_{1,0})^\perp \to (E^{>0}_{0,1})^\perp
$$

\begin{lemma}
 The term being subtracted in the Forni formula can be expressed as
 $$
 \sum_{i,j=g-k+1}^g |B_\omega(c_i,c_j)|^2 = \tr (\sigma_{22}\sigma_{22}^\dag)
 $$
\end{lemma}
\begin{proof}
 Choose a real, Hodge-orthogonal basis $c_1,\ldots, c_g$ of a Lagrangian space as in the assumptions of the formula.
 This gives a basis for $E^{1,0}$ by taking the holomorphic representatives $h(c_i)$ and a basis for $E^{0,1}$ by $\conj{h(c_i)}$.
 
 The matrix elements of $\sigma$ for these bases are given by $B_\omega(c_i,c_j)$.
 This implies the formula.
\end{proof}

\begin{corollary}
\label{cor:sigma_phi_perp}
 Suppose that the cocycle has $2k$ zero exponents.
 Then for any vectors $\phi_i,\phi_j\in (E^{>0}_{1,0})^\perp$ we have
 $$
 \ip{\sigma \phi_i}{\conj{\phi_j}}=0
 $$
 Here $\ip{-}{-}$ denotes the Hodge inner product.
\end{corollary}

The corollary follows since the sum of the first $g-k$ exponents already gives the sum of all the exponents.
Therefore the corresponding term in the formula vanishes pointwise.

\subsection{Proof of \autoref{prop:g_t-invar}}

All the considerations are along a fixed $g_t$-orbit for which the Oseledets theorem holds.
Recall we have the decomposition which is not $g_t$-invariant
$$
E = E^{<0}\oplus E^{\geq 0}_{F^1}\oplus E^{\geq 0 }_{\conj{F^1}} \oplus E^{>0}
$$
Recall the middle terms were defined by $E^{\geq 0}_{F^1}:= (E^0\oplus E^{>0})\cap F^1$ and its complex-conjugate version.

\begin{lemma}
The space $E^{\geq 0 }_{F^1}$ coincides with $(E^{>0}_{1,0})^{\perp}$ from \autoref{eqn:E_1001perp}.
\end{lemma}
\begin{proof}
 Recall $E^{\geq 0}_{F^1}=E^{\geq 0}\cap F^1$.
 Now $E^{\geq 0}$ is the same as the symplectic orthogonal of $E^{>0}$.
 Take a real class $c\in E^{>0}$ and decompose it according to $(1,0)$ and $(0,1)$ types
 $$
 c=\alpha\oplus \conj{\alpha}
 $$
 Take $\beta \in E^{\geq 0}\cap F^1$, then $(c,\beta)=0$ where $(,)$ is the symplectic pairing.
 Therefore $\alpha$ and $\beta$ are Hodge-orthogonal, for all $\alpha$ coming from a real class $c\in E^{>0}$.
 But these $\alpha$ will span $E^{>0}_{1,0}$, so $E^{\geq 0}_{F^1}$ is the subspace of the $(1,0)$ decomposition which is Hodge-orthogonal to $E^{>0}_{1,0}$.
\end{proof}

The above result implies that \autoref{cor:sigma_phi_perp} can be applied to elements of $E^{\geq 0}_{F^1}$.
For the computation, choose a Hodge orthonormal trivialization along a small piece of the orbit as follows:
\begin{enumerate}
 \item $c_1,\ldots, c_{g-k}$ are a basis of $E^{>0}$
 \item $\phi_{g-k+1},\ldots,\phi_g$ are a basis of $E^{\geq 0}_{F^1}$
 \item $\conj{\phi_{g-k+1}},\ldots,\conj{\phi_g}$ are conjugates of the preceding ones, and thus a basis of the conjugate subspace
\end{enumerate}

Recall from \autoref{lemma:dir_sum_orthog} that the decomposition
$$
E^{>0}\oplus E^{\geq 0}_{F^1}\oplus E^{\geq 0}_{\conj{F^1}}
$$
is Hodge-orthogonal.
To show that $E^{>0}\oplus E^{\geq 0}_{F^1}$ is $g_t$-invariant, it suffices to check that $\nabla^{GM}(E^{\geq 0}_{F^1})$ is perpendicular to $E^{\geq 0}_{\conj {F^1}}$.
Here $\nabla^{GM}$ denotes the Gauss-Manin connection.
We already know that $E^{>0}$ is invariant by Gauss-Manin, and so is $E^{\geq 0}$.

Now recall (see for example \cite[eqn. 3.4.1]{sfilip_ssimpleKZ}) the relation between the Gauss-Manin and Chern connections
$$
\nabla^{GM} = \nabla^{Ch} + \sigma + \sigma^\dag
$$
The Chern connection $\nabla^{Ch}$ preserves the $(p,q)$ type of forms, and $\sigma^\dag$ annihilates $(1,0)$-forms.
This implies that
$$
\nabla^{GM}\phi_i = \nabla^{Ch}\phi_i +\sigma \phi_i
$$
But if we take the inner product with $\conj{\phi_j}$, the first term on the right vanishes for type reasons.
The second term vanishes as a consequence of the Forni formula, i.e. \autoref{cor:sigma_phi_perp} gives
$$
\ip{\sigma \phi_i}{\conj{\phi_j}} = 0
$$
This implies $\ip{\nabla^{GM}\phi_i}{\conj{\phi_j}}=0$, which is what we wanted.
\hfill \qed

\subsection{Flatness of the refinement}

The $g_t$-invariance obtained can be improved, using the following result from \cite[Cor. 4.5]{EM}.
It is valid in a more general partially hyperbolic setting, see e.g. the work of Avila and Viana \cite{Avila_Viana_Invariance}.
The original idea goes back to Ledrappier \cite{Ledrappier}.

%  Let $(M,\mu,g_t)$ be a manifold equipped with a finite ergodic invariant measure under a flow.
%  Suppose that $\cF$ is an unstable for $g_t$ foliation on $M$.
%  Denoting by $\cF(x)$ the leaf of $\cF$ through $M$, unstable means that
%  $$
%  \forall y\in \cF(x) \textrm{ we have } \dist(g_{t}x,g_ty) \xrightarrow[t\to -\infty]{} 0
%  $$
%  
%  Suppose now that $E\to M$ is an integrable linear cocycle over $g_t$, and assume it has only one Lyapunov exponent $\lambda$. 
%  Denoting the fiber of $E$ over $x$ by $E_x$, this means that for $\mu$-a.e. $x\in M$ and $\forall v\in E_x$ we have
%  $$
%  \lim_{t\to \pm \infty} \frac 1t \log \norm{g_t v}= \lambda
%  $$
%  
%  Suppose that the cocycle $E$ is equipped with a $g_t$-covariant connection on the leaves of $\cF$.
%  This means that for all $y\in \cF(x)$ and a choice of path $x\to y$ in $\cF(x)$ we have linear operators $A_{x\to y}:E_x\to E_y$ which compose naturally to give a flat connection, and moreover
%  $$
%  g_t A_{x\to y}(v) = A_{g_t x\to g_t y}(g_t v)
%  $$
%  
%  Then if $F\subset E$ is a $g_t$-invariant subbundle, it must be flat on the leaves of $\cF$ for the connection $A$.

\begin{lemma}
 \label{lemma:unstable_flat}
 Suppose $M$ is a $g_t$-invariant subbundle of the KZ cocycle, or some piece thereof denoted $E$ (see \cite[Prop. 4.4]{EM} for the more general setting).
 Suppose further that for the Oseledets filtration $E^{\geq \bullet}$ on $E$, we have
 $$
 E^{\geq\lambda_k}\subsetneq M \subsetneq E^{\geq \lambda_{k+1}}
 $$
 with $\lambda_k>\lambda_{k+1}$.
 Then $M$ is flat along the unstable leaves of $g_t$.
\end{lemma}

Now we combine the above lemma with \autoref{prop:g_t-invar}.

\begin{proposition}
\label{prop:unstable_F1}
 In the notation of the previous section, consider the subbundle $E^{>0}\oplus E^{\geq 0}_{F^1}$. 
 Then it is flat along a.e. unstable leaf of the \Teichmuller geodesic flow $g_t$.
\end{proposition}
\begin{proof}
 By \autoref{prop:g_t-invar}, the subbundle is $g_t$-invariant.
 Moreover, it is a refinement of the backwards Oseledets filtration:
 $$
 E^{>0}\subsetneq \left( E^{>0}\oplus E^{\geq 0}_{F^1} \right)\subsetneq E^{\geq 0} \subset E
 $$
 It also induces a refinement of the quotient bundle $E^{\geq 0}/ E^{>0}$.
 By \autoref{lemma:unstable_flat} we see that the induced filtration on the quotient is flat along the unstable leaves.
 Since $E^{>0}$ is also flat along unstable leaves, we get the desired conclusion. 
\end{proof}

\subsection{Restrictions on monodromy}

To get further information, we now use the holomophic properties of the subbundle $F^1:=H^{1,0}$ of the Hodge bundle.
The above argument applied to $\SL_2\bR$-invariant bundles, but now we restrict to flat pieces of the semisimple decomposition.

We work locally in period coordinates, given by $\bC^N=\bR^{2N}$ and denote points as $(x,y)$ to distinguish the real and imaginary parts.
We have some local chart $U\subset \bC^N$ and pick an Oseledets-generic point $(x_0,y_0)$ to which the above results apply.

The Oseledets theorem gives a filtration at $(x_0,y_0)$, which we refined using the previous results:
\begin{align}
\label{eq:filtration}
E^{>0}_{(x_0,y_0)}\subset \left( E^{>0}\oplus E^{\geq 0}_{F^1}\right)_{(x_0,y_0)} \subset E^{\geq 0}_{(x_0,y_0)}\subset E_{(x_0,y_0)}
\end{align}
We view the Hodge filtration $F^1$ as a holomorphic map into the Grassmanian
$$
F^1:U\to \Gr_g(E_\bC)
$$
Consider the locus of points where $F^1$ intersects the spaces from \eqref{eq:filtration}, extended to the neighborhood using the flat connection:
$$
Z:=\left\lbrace (x,y)\in U : \dim\left(F^1(x,y)\cap \left( E^{>0}\oplus E^{\geq 0}_{F^1}\right)_{(x_0,y_0)} \right)\geq k  \right\rbrace
$$
Here $2k$ is the number of zero exponents of the KZ cocycle.

\begin{proposition}
 The locus $Z$ is all of $U$.
\end{proposition}
\begin{proof}
 The locus $Z$ is a closed complex-analytic subset of $U$, since it is given by the vanishing of holomophic functions.
 Moreover, it is invariant under the natural scaling action of $\bC^\times$ on $U\subset \bC^N$, since the Hodge filtration $F^1$ is, while the space $\left( E^{>0}\oplus E^{\geq 0}_{F^1}\right)_{(x_0,y_0)} $ is fixed in the argument.
 
 \autoref{prop:unstable_F1} implies that $Z$ contains the unstable leaf through $(x_0,y_0)$.
 Indeed, $F^1(x,y)$ intersects $\left( E^{>0}\oplus E^{\geq 0}_{F^1}\right)_{(x,y)}$ by construction;
 but on the fixed unstable leaf, the second space agrees with the fixed space $\left( E^{>0}\oplus E^{\geq 0}_{F^1}\right)_{(x_0,y_0)}$.
 
 The unstable leaf is the set of points of the form $(x_0+v,y_0)$ where $\omega(v,y_0)=0$ with $\omega$ the symplectic pairing.
 This implies that $Z$ contains a $\bC^{N-1}$ of the form
 $$
 (x_0+v,y_0+v') \textrm{ where } \omega(v,y_0)=\omega(v',y_0)=0
 $$
 Indeed, if a holomophic function vanishes on $\bR^{N-1}\subset \bC^{N-1}$, it vanishes on all of $\bC^{N-1}$ (consider the power series expansion).
 
 Because the $\bC^{N-1}$ we obtained is transverse to the $\bC^\times$-action, we find that $Z$ is all of $U$.
\end{proof}

So for a generic $(x_0,y_0)$ the spaces $\left( E^{>0}\oplus E^{\geq 0}_{F^1}\right)_{(x_0,y_0)}$ and $F^1(x,y)$ intersect in dimension at least $k$, for all $(x,y)$ near $(x_0,y_0)$.

\begin{proposition}
 For all $(x,y)$ near $(x_0,y_0)$ as above, $F^1(x,y)$ intersects in dimension at least $k$ the images of $\left( E^{>0}\oplus E^{\geq 0}_{F^1}\right)_{(x_0,y_0)}$ under the monodromy group.
\end{proposition}
\begin{proof}
 Pick a loop $\gamma$ in $\cM$ starting at $(x_0,y_0)$ and let $\rho(\gamma)$ be the corresponding monodromy matrix in $E$.
 By analytic continuation, $F^1$ and the (flat) parallel transport of $\left( E^{>0}\oplus E^{\geq 0}_{F^1}\right)_{(x_0,y_0)}$ along $\gamma$ must also intersect.
 But when $\gamma$ returns to $(x_0,y_0)$ we find that $\left( E^{>0}\oplus E^{\geq 0}_{F^1}\right)_{(x_0,y_0)}$ changed to $\rho(\gamma)\left( E^{>0}\oplus E^{\geq 0}_{F^1}\right)_{(x_0,y_0)}$.
\end{proof}

\begin{corollary}
\label{corollary:Zariski}
 Let $G$ be the Zariski closure of the monodromy of $E$ and let $(x,y)$ be near $(x_0,y_0)$.
 Then $\forall g\in G$ we have that $F^1(x,y)$ intersects $g\cdot \left( E^{>0}\oplus E^{\geq 0}_{F^1}\right)_{(x_0,y_0)}$ in dimension at least $k$.
\end{corollary}
This follows since the condition of intersection is Zariski-closed.
It also holds for both the real and complex Zariski closures.

\section{Classifying the monodromy}
\label{sec:monodromy}

In \autoref{subsec:real_rk_zero_exp} we analyze (in the abstract algebraic setting) restrictions on monodromy given by \autoref{corollary:Zariski}.

\subsection{Real rank and zero exponents}
\label{subsec:real_rk_zero_exp}

We consider a real semisimple algebraic group $G$ acting on a real vector space $E_\bR$ of dimension $2g$, preserving a symplectic form.
After complexifying, assume we have a decomposition $E_\bC=F^1\oplus \conj{F^1}$ where $F^1$ is a Lagrangian subspace.
Moreover, for the symplectic pairing we have $\sqrt{-1}(\alpha,\conj{\alpha})>0$ for all $\alpha\in F^1$.

Following the conclusion of \autoref{corollary:Zariski}, we assume there exists a Lagrangian $L\subset E_\bC$ such that $\forall g\in G(\bC)$ we have
$$
\dim \left((g\cdot L) \cap F^1\right)\geq k
$$
Moreover, we have a subspace $L_0\subset L$ defined over $\bR$ and of dimension $g-k$.
Here $L$ plays the role of the space $\left( E^{>0}\oplus E^{\geq 0}_{F^1}\right)_{(x_0,y_0)}$ and $L_0$ of its subspace $E^{>0}_{(x_0,y_0)}$.
We assume $L_0$ is isotropic and defined over $\bR$, and $L$ is Lagrangian.

Let $A\subset G$ be a maximal real split torus, i.e. a subgroup of maximal dimension isomorphic to $\bR^\times$ to some power.
We can consider the weight space decomposition of $E_\bR$ for the action of this torus, i.e. the eigenvalues which can occur.
These are viewed as elements of the dual of the Lie algebra of $A$.

\begin{proposition}
\label{prop:real_rk}
 Under the above assumptions, the weights of $E_\bR$ with respect to the action of $A$ contain zero with multiplicity at least $2k$.
\end{proposition}
\begin{proof}
 \textbf{Step 1:} We can assume that $A$ fixes $L$ as a subspace of $E_\bC$.
 
 Indeed, consider the action of $A$ on the orbit closure $\closure{A\cdot L}$ inside the Grassmanian of $E_\bC$.
 The closure is for the Zariski topology, and we view it as an $\bR$-algebraic variety via Weil restriction of scalars.
 Now the Borel fixed point theorem applies, see \cite[14.1.7]{Springer} for the case of split solvable groups, without assuming an algebraically closed field.
 
 So there exists a fixed point under the action of $A$ and we can assume this is our $L$.
 The dimension of intersection with $F^1$ did not drop.
 
 \noindent\textbf{Step 2:} We show $A$ acts trivially on $L\cap F^1$, which is of dimension at least $k$.
 This will suffice, since then $A$ also acts trivially on its complex conjugate.
 Therefore $A$ acts trivially on a $2k$-dimensional space.
 
 To prove the assertion, note that $L_0\subset L$ is preserved by $A$.
 Indeed, we have that $L_0=L\cap \conj{L}$ and both spaces are preserved by $A$, since $A$ is a real torus.
 
 Next, all of $G$ and $A$ in particular preserves the pseudo-hermitian form on $E_\bC$, denoted by $h$.
 By the isotropy properties of our subspaces, we see that $L_0$ is in the radical $\Rad(h\left. \right|_L)$ of $h$ restricted to $L$.
 On $L\cap F^1$ the pseudo-hermitian form $h$ is positive definite, and we have $L=L_0\oplus (L\cap F^1)$.
 
 Consider now the weight decomposition of $L$ with respect to $A$, possible since $L$ is $A$-invariant.
 
 By contradiction, assume there exists a non-trivial weight space outside of $L_0$ and take a vector $v\neq 0$ in it.
 Then $h(v,v)>0$.
 However, there exists a $1$-parameter subgroup $\{a_t\}$ in $A$ such that
 $$
 \lim_{t\to \infty} a_t v =0
 $$
 But $a_t$ preserves the pseudo-hermitian norm of $v$ and this is a contradiction.
\end{proof}

\subsection{Classifying the groups in the algebraic hull}
\label{sec:alg_hull_classification}

In this section, we describe the possible semisimple Lie groups which arise in irreducible (over $\bR$) pieces of the monodromy.
To this end, we work with a general variation of Hodge structures $E$ of weight $1$ over a quasi-projective base $M$.
Let $\Gamma\subset \GL(E_\bZ)$ be the monodromy group and $G:=\closure{\Gamma}\subset \GL(E_\bR)$ the Zariski closure of the monodromy.
Then $G$ is a real semisimple algebraic group, and let $G^\circ$ be the connected component of the identity.

Choose now some $x\in M$.
Then we have a Hodge structure on the fiber $E_x$ of $E$ over $x$, and it has a corresponding Mumford-Tate group $\MT(E_x)$.
More details can be found, for instance, in \cite[Section 3]{DMOS_Deligne}.

One possible definition of the Mumford-Tate group is as follows.
The Hodge structure on $E_x$ is given by a representation of the Deligne torus $h:\bS\to \GL(E_x)$ (see \cite{Deligne_travaux} for this point of view).
Then $\MT(E_x)$ is the smallest $\bQ$-algebraic group which contains the image of $\bS$.
It is a reductive $\bQ$-algebraic group.
The relevant fact for us is the next result (\cite[Prop. 7.5]{Deligne_K3}, see also \cite[Thm. 1]{AndreMTgps}).

\begin{proposition}
 For a generic $x\in M$ (outside a countable collection of proper analytic subsets) we have that $G^\circ\normal \MT(E_x)$.
 In other words, the connected component of the identity of $G$ is a normal subgroup of the Mumford-Tate group.
 
 An equivalent statement is that after passing to a finite cover, the Zariski closure of the monodromy is a normal factor of the Mumford-Tate group.
\end{proposition}

Therefore, in order to list the possibilities for $G^\circ$, we need the possible real semisimple subgroups of the Mumford-Tate group of a Hodge structure of weight $1$.
Consider the action of $\MT(E_x)$ on $E_x$ and its decomposition into $\bR$-irreducible factors, and take one such, denoted again by $G$.
Then $G$ is an almost direct product of $\bR$-simple factors and its center:
$$
G=G_1\cdots G_n Z
$$
Its action on $E_x$ is given by some representation $\rho = \rho_1\boxtimes\cdots\boxtimes \rho_n\boxtimes \chi$, where $\rho_i$ is a representation of $G_i$ and $\chi$ of $Z$.
The Hodge structure is given by a map $h:\bS\to G$, which we can also project to (a finite quotient of) each $G_i$.
Denote the projections by $p_i:G\to G_i'$, where $G_i'$ equals $G_i$ mod a finite central subgroup.

\begin{proposition}
 If $\rho\circ h$ gives a Hodge structure of weight $1$, then $h$ is non-trivial on at most one non-compact factor of $G$.
 Moreover, $p_i\circ h$ is trivial on the compact factors. 
\end{proposition}
\begin{proof}
 Consider the induced Hodge structures given by $\rho_i \circ p_i \circ h$.
 If they only have type $(0,0)$, this implies $p_i\circ h$ is in the finite center of $G_i'$, thus trivial.
 This can only occur for compact factors.
 
 There can be only one non-trivial representation leading to weights of type $(1,0)$ and $(0,1)$.
 Otherwise, the tensor product of two such will have at least three non-trivial weight spaces.
\end{proof}

In other words, we only need to consider representations $h:\bU\to G$ where $G$ is a non-compact simple real algebraic group.
Here, $\bU$ is the unit circle group, a natural subgroup of $\bS$.
Indeed, the compact factors only give rise to multiplicities, while the scaling is accounted for by the center.

To summarize, we want to classify representations
\begin{align}
\label{eqn:Hodge_rep}
\bU\xrightarrow{h} G \xrightarrow{\rho} \GL(E)
\end{align}
where $\bU=S^1$ is the unit circle group, $G$ is an $\bR$-simple group and $\rho$ is a representation such that $\rho\circ h$ yields a Hodge structure of weight $1$.

To begin the analysis, consider the adjoint representation of $G$ on its Lie algebra $\frakg$.
The representation $h$ gives a Hodge structure on $\frakg$ of weight $0$ and a decomposition
$$
\frakg = \frakg^{1,-1}\oplus \frakg^{0,0}\oplus \frakg^{-1,1}
$$
The classification of such Lie algebras is identical to that of Hermitian symmetric domains (see \cite[Ch. X, Sect. 6.1]{Helgason}).
It reduces to classifying pairs (Dynkin diagram, special vertex) where a vertex is special if the corresponding highest weight is minuscule.
Terminology follows \cite{Bourbaki} and a good introduction is in the lecture notes of Looijenga \cite[Sect. 2]{Looijenga}.
Following \cite{Helgason} the list is:
\begin{enumerate}
 \item[(i)] $A_n$ and any vertex. 
 Then $G=\SU_{p,q}$.
 \item[(ii)] $B_n$ and the end vertex corresponding to a long root.
 Then $G=\SO_{2n-1,2}(\bR)$.
 \item[(iii)] $C_n$ and the end vertex corresponding to the (unique) long root.
 Then $G=\Sp_{2g}(\bR)$.
 \item[(iv)] $D_n$ and either of the three end vertices.
 Then\footnote{The Lie algebra of $\SO_{2n}^*$ is described and analyzed in \autoref{eg:so*} } $G=\SO_{2n}^*$ or $\SO_{2n-2,2}(\bR)$.
 \item[(v)] Exceptional groups, but as we shall see, these do not pass the next test.
\end{enumerate}

The condition that $\bU\to G$ admits a representation with Hodge structure of weight $1$ implies that $G$ must be classical and the highest weight of the representation is minuscule.
We refer for this to \cite[Thm. B.65]{Lewis}, but see also \cite{Serre}.
The list of simple pieces of the representation after complexification is:
\begin{enumerate}
 \item[(i)] $\fraksl_n\bC$ in any exterior power of the standard representation.
 \item[(ii)] $\so_{2n+1}\bC$ in the spin representation.
 \item[(iii)] $\fraksp_{2g}\bC$ in the standard representation.
 \item[(iv)] $\so_{2n}\bC$ in the standard representation, or either of the spin representations.
\end{enumerate}

Inspecting the first list to see which of the allowed representations give a Hodge structure of weight $1$, we find
\begin{enumerate}
 \item [(i)] $\su_{p,q}$ in the standard representation, or $\su_{p,1}$ in any exterior power representation.
 \item [(ii)] $\so_{2n-1,2}(\bR)$ in the spin representation.
 \item [(iii)] $\fraksp_{2g}(\bR)$ in the standard representation.
 \item [(iv)] $\so_{2n}^*$ in the standard representation, or $\so_{2n-2,2}(\bR)$ in either of the spin representations.
\end{enumerate}

Note that this list is essentially the same as that given by Satake in \cite[pg. 461]{Satake_classification}.
See also the exposition of Deligne \cite[Ch. 1]{Deligne_VarShim}, in particular Table 1.3.9 and Remark 1.3.10.

Let us summarize our findings.
\begin{theorem}
\label{thm:gps_alg_hull}
 From \cite{sfilip_ssimpleKZ} the Kontsevich-Zorich cocycle is semisimple, and its decomposition respects the Hodge structure.
 Consider an $\bR$-irreducible piece, and let $G$ be the corresponding semisimple group in the algebraic hull.
 
 Then $G$ has at most one non-compact factor, and it lies in a certain representation.
 At the level of Lie algebras, the corresponding real Lie algebra and representation must be one from the list (repeated from above)
\begin{enumerate}
 \item [(i)] $\su_{p,q}$ in the standard representation, or $\su_{p,1}$ in any exterior power representation.
 \item [(ii)] $\so_{2n-1,2}(\bR)$ in the spin representation.
 \item [(iii)] $\fraksp_{2g}(\bR)$ in the standard representation.
 \item [(iv)] $\so_{2n}^*(\bR)$ in the standard representation, or $\so_{2n-2,2}(\bR)$ in either of the spin representations.
\end{enumerate}
\end{theorem}
 
\begin{remark}
 Note that a Lie algebra $\frakg$ is in the list if there is a corresponding real Lie group $G$, with Lie algebra $\frakg$ with the following property.
 There is a homomorphism from the unit circle $h:\bU\to G$, and a representation of $G$ denoted $\rho:G\to \GL(E)$ such that $\rho\circ h$ endows $E$ with a Hodge structure of weight $1$.
 This is described in \autoref{eqn:Hodge_rep}, and is the only thing necessary for the classification.
\end{remark}

The classification in \autoref{thm:gps_alg_hull} is ultimately concerned only with Hodge structures of weight $1$.
Therefore, it applies to both the flat and $\SL_2\bR$-invariant semisimple decompositions of the Kontsevich-Zorich cocycle, as the next result shows.

\begin{theorem}
 \label{thm:SL_2_invar}
 Let $E$ be an irreducible piece of the $\SL_2\bR$-invariant decomposition of the Kontsevich-Zorich cocycle.
 Let $G$ be the algebraic hull of $E$, described as follows.
 Consider the list of all $\SL_2\bR$-invariant tensors $\{\tau_{i,j,k}\}$ with $\tau_{i,j,k}\in E^{\otimes i}\otimes (E^\dual)^j$ a global section (e.g. the symplectic form will always be in the list).
 The algebraic hull $G$ in a fiber $E_x$ of the bundle is the subgroup of $\GL(E_x)$ which preserves all these tensors.
 
 Then the Lie algebra $\frakg$ of $G$ must be one from the list in \autoref{thm:gps_alg_hull}.
\end{theorem}
\begin{proof}
 Each fiber $E_x$ carries a Hodge structure of weight $1$, thus a homomorphism $h:\bU\to \GL(E_x)$.
 To prove the result, it suffices to check that the image of $\bU$ lands inside $G$ (c.f. \autoref{eqn:Hodge_rep}).
 
 \paragraph{Step 1.}
 In the list of tensors $\tau_{i,j,k}$ defining $G$, we can select a subset of tensors which are of pure Hodge type $(n,n)$ (in each fiber of the tensor bundle) and still recover the same group.
 
 Indeed, by the results from \cite{sfilip_ssimpleKZ}, any $\SL_2\bR$-invariant tensor must have all its $(p,q)$-components in the Hodge decomposition also $\SL_2\bR$-invariant.
 Moreover, any invariant section will have an invariant complement, respecting the Hodge structure.
 Thus projection onto the invariant section defines another tensor, of pure type $(0,0)$ if the section was of pure type $(p,q)$.
 
 Finally, the section is invariant up to scaling if and only if projection onto its subspace is invariant.
 But the cocycle cannot act by scaling on a section of pure type $(p,q)$ since it must preserve its indefinite norm (non-trivial, since the section is pure).
 
 \paragraph{Step 2.}
 The algebraic hull $G$ is described as a group preserving a list of tensors of type $(n,n)$.
 However, the image $h(\bU)$ preserves all tensors of type $(n,n)$, and this is its defining property.
 Therefore, it must be contained inside $G$.
\end{proof}

\section{Geometry of the Oseledets theorem}
\label{sec:overview_Oseledets}

This section describes the geometric interpretation of the Oseledets theorem and how it relates to the usual ``matrix'' version.
The geometric version from \autoref{sec:geom_oseledets}, following Kaimanovich \cite{Kaimanovich} becomes a statement about drift in symmetric spaces.
\autoref{sec:examples} contains examples relevant to the KZ cocycle.

\subsection{Oseledets theorem in semisimple Lie groups}
\label{sec:geom_oseledets}
Let $G$ be a group and suppose $\alpha$ is $G$-valued cocycle over a space $M$ equipped with an ergodic measure $\mu$ invariant under a flow $g_t$.
Choosing a representation $\rho$ of $G$ leads to a linear cocycle and the usual Oseledets theorem applies to it.
In this section, following Kaimanovich \cite{Kaimanovich} we describe the geometric version of the Oseledets theorem and how it relates to the Lyapunov exponents of linear cocycles.

\paragraph{Structure theory of $G$.}
We refer to \cite[Ch. 28, 29]{Bump} for facts and terminology used below.

Assume that $G$ is a real semisimple Lie group, with Lie algebra $\frakg$.
Then it has a maximal compact $K\subseteq G$ with Lie algebra $\frakk$.
The corresponding Cartan involution $\sigma$ leads to a decomposition into $\pm 1$ eigenspaces $\frakg = \frakk \oplus \frakp$.

Inside $\frakp$ choose a maximal abelian subalgebra $\fraka$ and decompose $\frakg$ for the adjoint action of $\fraka$.
The weights form a (not necessarily reduced) root system $\Phi_\fraka\subset \fraka^\dual$.

Extend now $\fraka$ to a maximal Cartan subalgebra $\frakh$ by adding a compact torus.
For the action of $\frakh$ on $\frakg_\bC$, we get the root system $\Phi_\frakg\subset \frakh^\dual$.
The inclusion $\fraka\subset \frakh$ yields a map $\frakh^\dual\to\fraka^\dual$ which induces a map on roots $\Phi_\frakg\to \Phi_\fraka$.
See Bump \cite[Ch. 29]{Bump} for the general theory, as well as \autoref{eg:su_pq}, \autoref{eg:orthogonal}, and \autoref{eg:so*} below.

Recall also that we have the associated symmetric space $X:=G/K$.
We also have the decomposition $G=KAK$, where $A$ is the subgroup corresponding to the abelian Lie algebra $\fraka$.
Geodesics in $X$ are described by a choice of a vector in $\fraka$, the ``speed'', and an element of $K$, the ``direction''.
We also have the positive Weyl chamber, denoted $\fraka^+\subset \fraka$.

\begin{example}
 Suppose that $G=\SL_n\bR$, then $K=\SO_n(\bR)$.
 The space $X=G/K$ is also the space of metrics on $\bR^n$, so a choice of point in $X$ is the same as a choice of metric on $\bR^n$.
 
 The abelian subalgebra $\fraka$ is the set of trace zero diagonal matrices.
 Therefore, the subgroup $A$ consists of diagonal matrices of determinant $1$ and positive entries.
 The positive Weyl chamber $\fraka^+\subset \fraka$ consists of trace zero diagonal matrices with entries arranged in increasing order.
\end{example}

Another illustration is given in \autoref{eg:su_pq} and \autoref{eg:orthogonal}.

For the next result, keep the notation as above.

\begin{theorem}[Oseledets theorem, \cite{Kaimanovich}]
 Suppose $(M,\mu,g_t)$ is a probability measure space with an ergodic flow.
 Fix a semisimplie Lie group $G$ and let $\cP\to M$ be a $G$-bundle over $M$, with a lift of the $g_t$-action to $\cP$.
 
 On the associated symmetric space bundle $\cX\to M$ with fiber $X=G/K$, choose a basepoint $h:M\to \cX$.
 Suppose the induced cocycle is integrable for the basepoint $h$, meaning that the function on $M$
 $$
 N(x):=\sup_{t\in [-1,1]} \dist(h(g_t x), g_t h(x))
 $$
 is $\mu$-integrable, where $g_t$ denotes the time-$t$ map of the flow.
 
 Then there exists a vector $\Lambda\in \fraka^+$, called the Lyapunov vector, with the following property.
 For $\mu$-a.e. $x\in M$, decompose the geodesic in $G/K$ from $h(g_T x)$ to $g_T h(x)$ in the $KAK$ decomposition as $k_1 A_T k_2$.
 We can take the logarithm of $A_T\in A$ to get an element of $\fraka$ and then
 $$
 \lim_{T\to \infty} \frac 1 T \log A_T = \Lambda
 $$
 There is also a corresponding statement for the convergence of directions, i.e. of the elements in $K$.
 It corresponds to the existence of Oseledets subspaces.
\end{theorem}
\begin{remark}
 Given a representation $\rho$ of $G$, it has corresponding weights (with multiplicities) $\Sigma_\rho\subset \frakh^\dual$.
 The Lyapunov exponents of the linear cocycle $\rho\circ\alpha$ are given by the evaluation
 $$
 \Sigma_\rho\into \frakh^\dual\to \fraka^\dual\xrightarrow{ev_\Lambda}\bR
 $$
 Namely, the real numbers which are the images of $\Sigma_\rho$ under the above map are the Lyapunov exponents.
 
 In the example of $\SL_n\bR$ discussed above, simplicity of the spectrum (for the standard representation) means that $\Lambda$ lies in the interior of the Weyl chamber $\fraka^+$.
\end{remark}

\begin{corollary}
\label{cor:zero_is_rk}
 The linear cocycle $\rho\circ \alpha$ has at least as many zero exponents as there are weights mapped to zero via the composition
 $$
 \Sigma_\rho\to \frakh^\dual\to \fraka^\dual
 $$ 
\end{corollary}

\subsection{Some examples}
\label{sec:examples}

We now work out specific examples that can potentially appear in irreducible factors of the KZ cocycle.
Below, let $I_n$ denote the identity matrix of size $n$ and $J_n$ the square matrix of size $n$ having ones on the antidiagonal and zeros elsewhere:
$$
I_n=\begin{bmatrix}
 1 & & \\
 & \ddots & \\
 & & 1
\end{bmatrix}
J_n = \begin{bmatrix}
       & & 1\\
       & \ddots & \\
       1 & & 
      \end{bmatrix}
$$

\begin{example}
\label{eg:su_pq}
 Let $\su_{p,q}$ be the Lie algebra preserving the indefinite hermitian form on $\bC^{p+q}$ given by
 $H=\left[\begin{smallmatrix}
   I_p & 0\\
   0 & -I_q
  \end{smallmatrix}
  \right]$
 Then we have that
 $$
 \su_{p,q}=\left\lbrace A\in \Mat_{n}\bC \left|
 A=\begin{bmatrix}
 a & b\\
 b^\dag & c
 \end{bmatrix}
 \right.
 \text{ with } a = -a^\dag, c=-c^\dag, \tr A = 0
 \right\rbrace
 $$  
 The Cartan involution is given by
 $$
 \sigma(X) = HXH
 $$
 The fixed points of $\sigma$ give the maximal compact subalgebra $\fraks(\fraku_p\times\fraku_q)$.
 
 To see the relative roots (assume $p\geq q$) conjugate $H$ to be in the form
 $$
 H_1 = 
 \begin{bmatrix}
  0 & 0 & I_q\\
  0 & I_{p-q} & 0\\
  I_q & 0 & 0
 \end{bmatrix}
 $$
 Then $\su_{p,q}$ becomes the set of matrices of the form
 $$
 \left\lbrace
 \begin{bmatrix}
  a & x & b\\
  y & u & -x^\dag \\
  c & -y^\dag & -a^\dag
 \end{bmatrix}
 \text{ with }
 b,c,u \text{ skew-Hermitian}
 \right\rbrace
 $$
 The maximal split abelian subalgebra $\fraka$ consists of diagonal matrices
 $$
 \fraka = \left\lbrace 
\begin{bmatrix}
t_1 &   &   &   &   &   &  \\
 & \ddots &   &   &   &   &  \\
 &   & t_q &   &   &   &  \\
 &   &   & 0_{p-q} &   &   &  \\
 &   &   &   & -t_q &   &  \\
 &   &   &   &   & \ddots &  \\
 &  &   &   &   &   & -t_1
\end{bmatrix}
 \right\rbrace
 $$
 If $f_i$ denotes a weight on which $\fraka$ acts by $t_i$, then in the description of $\su_{p,q}$ above we have that
 \begin{enumerate}
  \item [On $a$] the roots are $f_i-f_j$ for $i \neq j$
  \item [On $b$] the roots are $f_i+f_j$ for $i\neq j$ and $2f_i$
  \item [On $c$] the roots are $-(f_i+f_j)$ for $i\neq j$ and $-2f_i$
  \item [On $x$] the roots are $f_i$
  \item [On $y$] the roots are $-f_i$
 \end{enumerate}
 This shows that the relative root system is of type $BC_q$.
 
 Next, recall that the root system of $\su_{p,q}$ is of type $A_{p+q-1}$.
 We can express its roots in the form
 $$
 \alpha_i = e_i - e_{i+1} \text{ for } i=1\ldots p+q-1
 $$
 where the $e_i$ are such that $e_1+\cdots + e_{p+q}=0$.
 To describe the map from roots to relative roots, recall the basis for $BC_q$:
 $$
 \lambda_i = f_i- f_{i+1} \text{ for } i=1\ldots q-1 \text{ and } \lambda_q = f_q
 $$
 Then, following \cite[Table 4]{Vinberg} we find that
 $$
 r(\alpha_j) = \lambda_j = r(\alpha_{p+q -j} ) \text{ for } j=1\ldots q-1 \text{ and } r(\alpha_j)=0 \text{ otherwise}
 $$
 Solving for $e_i$ and $f_i$, we find that
 \begin{align*}
  e_1\mapsto f_1 && e_{p+q}\mapsto -f_1 \\
  \vdots & & \vdots\\
  e_q\mapsto f_q &&e_{p+1}\mapsto -f_q
 \end{align*}
 and the rest of the $e_i$ map to zero.
 
 The standard representation of $\su_{p,q}$ has real dimension $2(p+q)$.
 After complexification it gives two copies of the standard representation of $\fraksl_n\bC$, with weights
 $$
 e_1,e_2,\cdots, e_{p+q}
 $$
 If we take its $k$-th exterior power over the complex numbers, the weights are of the form
 $$
 \{e_{i_1}+\cdots + e_{i_k} \vert \text{the }i_j \text{ are distinct} \}
 $$
 
 Apply now the map $r$ to get the restricted weights.
 For the standard representation, we get two copies of the weights
 $$
 \{f_1, \ldots, f_q, 0,\ldots,0,-f_q,\ldots,-f_1\}
 $$
 where $0$ occurs with multiplicity $p-q$.
 
 For the $k$-th exterior power over $\bC$, we find a similar structure, where $0$ also occurs.
 Its multiplicity is given by
 $$
 \sum_{a\geq 0} \binom{q}{a }\cdot \binom{p-q}{k-2a}
 $$
 The first binomial coefficient corresponds to the choice of nontrivial weights $f_i$ (and necessarily $-f_i$).
 The second corresponds to the remaining choice of vanishing weights. 
\end{example}

\begin{remark}
\label{rmk:su_pq_index}
 A consequence of the above computation is the following connection between the restricted weights and the pseudo-hermitian metric on exterior power representations.
 
 In any exterior power (over $\bC$) of the standard representation of $\su_{p,q}$ the action preserves a pseudo-hermitian metric.
 In the restricted weight decomposition, the non-zero weights come in pairs $\alpha$ and $-\alpha$.
 Each such weight space is isotropic, but the pseudo-hermitian form restricted to their span has signature $(1,1)$.
 
 The zero restricted weight spaces are of two kinds.
 Recall that a weight $e_{i_1}+\cdots + e_{i_k}$ gets mapped to zero in the restricted weight space decomposition, if and only if whenever $e_i$ occurs, so does $e_{p+q+1-i}$, for all $i=1\ldots q$.
 Call these ``canceling pairs''.
 
 The first type of weight which gets restricted to zero has an even number of canceling pairs, and so the indefinite form will be positive-definite on it.
 The second type involves an odd number of canceling pairs, so the indefinite form will be negative-definite on it.
 See \autoref{prop:su_p1_details} for some explicit computations.
 
 A dynamical consequence is related to \cite[Prop. 3.2]{FMZ_zero}.
 Namely, even if the KZ cocycle has a piece with $\SU_{p,q}$ in an exterior power representation, the geodesic flow $g_t$ preserves an appropriate metric on the zero Oseledets subspaces.
\end{remark}

\begin{example}
 \label{eg:orthogonal}
 We now describe the orthogonal groups in the spin representations.
 The maps between relative roots systems are from \cite[Table 4, pg. 231]{Vinberg}.
 The multiplicities of the relative roots are also available in loc.cit.
 
 Before proceeding, recall the description of the root systems of type $B_n$ and $D_n$.
 We provide a basis for the root system (i.e. the simple roots), as well as the fundamental dominant weights.
 We describe them in a vector space of dimension $n$, with a chosen basis $e_1,\ldots,e_n$. 
 
 \centerline{\textbf{Type $B_n$}}
 \begin{tabular}{lll}
  Simple roots & $\alpha_i = e_i-e_{i+1}$ 	& for $ i=1\ldots n-1$\\
	       & $\alpha_n = e_n$ 		& \\
  Fund. weights& $\varpi_i=e_1+\cdots+e_i$ 	& for $i=1\ldots n-1$\\
	      & $\varpi_n = \frac 12 (e_1+\cdots +e_n)$ 	&
 \end{tabular}
\vskip 0.3cm 
 \centerline{\textbf{Type $D_n$}}
 \begin{tabular}{lll}
  Simple roots & $\alpha_i = e_i-e_{i+1}$ 	&\hspace{-1em} for $ i=1\ldots n-1$\\
	       & $\alpha_n = e_{n-1}+e_n$ 	& \\
  Fund. weights& $\varpi_i=e_1+\cdots+e_i$ 	&\hspace{-1em} for $i=1\ldots n-2$\\
		&$\varpi_{n-1} = \frac 12 (e_1+\cdots+e_{n-1} -e_n)$ & \\
		&$\varpi_n = \frac 12 (e_1+\cdots+e_{n-1} +e_n)$ &
 \end{tabular}
\vskip 0.3cm 
\noindent We now apply this to specific real Lie algebras.
\vskip 0.3cm
\noindent\textbf{Type $B_n$, algebra $\frakg=\so_{2n-1,2}\bR$.}
 We view this orthogonal algebra as the one preserving the indefinite form given by the matrix
 $$
 \begin{bmatrix}
  I_{2n-3} & 0 \\
  0 & J_4
 \end{bmatrix}
 \text{ where }
 J_4=\begin{bmatrix}
    0 & 0 & 0 & 1\\
    0 & 0 & 1 & 0\\
    0 & 1 & 0 & 0\\
    1 & 0 & 0 & 0\\
   \end{bmatrix}
 \text{ and } I_{2n-3} \text{ is an identity matrix}
 $$
 The maximal split subalgebra $\fraka$ is given by matrices of the form
 $$
 \begin{bmatrix}
  0_{2n-3} & 0\\
  0 & T
 \end{bmatrix}
 \text{ where }
 T=
 \begin{bmatrix}
  t_1 & 0 & 0 & 0\\
  0 & t_2 & 0 & 0\\
  0 & 0 & -t_2 & 0\\
  0 & 0 & 0 & -t_1\\
 \end{bmatrix}
 $$
 The weights that occur for the action of $\fraka$ on $\frakg$ are the restricted weights and form a root system of type $B_2$.
 Call the simple roots $\lambda_1,\lambda_2$ (written in a basis $\lambda_1 = f_1-f_2, \lambda_2=f_2$).
 The full root system of $\frakg$ is of type $B_n$, call the simple roots $\alpha_1,\ldots,\alpha_n$ (described in a basis $e_i$ above).
 
 According to \cite[Table 4]{Vinberg}, the induced morphism $r:\frakh^\dual\to\fraka^\dual$ of root spaces is
 $$
 r(\alpha_i)=\lambda_i \text{ for } i=1,2 \text{ and } r(\alpha_i)=0 \text{ otherwise}
 $$
 Solving for $e_i$ these equations, we find
 \begin{align*}
 r(e_1)&=\lambda_1+\lambda_2 = f_1\\
 r(e_2)&=\lambda_2 = f_2\\
 r(e_j)&=0 	\text{ otherwise}
 \end{align*}
 
 Next, recall that the standard representation of $\so$ has highest weight $\varpi_1$ and is thus denoted $V(\varpi_1)$.
 The weights that occur in it are therefore
 $$
 \Sigma_{V(\varpi_1)} = \{ e_1,\ldots,e_n,0,-e_n,\ldots,-e_1\}
 $$
 The weights that occur in the restricted root system of $\so_{2n-1,2}$ are the ones which give the Lyapunov spectrum.
 To find them, apply the map $r$ described above to the set of weights to find:
 $$
 \{f_1,f_2,0,\ldots,0,-f_2,-f_1\}
 $$
 
 For future purposes, we also need the spin representation $V(\varpi_n)$.
 From \cite[Thm. 31.2]{Bump}, it has dimension $2^n$ and weights
 $$
 \left\lbrace\frac 12 \left(\pm e_1\pm e_2 \pm \cdots\pm e_n\right)\right\rbrace
 $$
 Again, to find the Lyapunov exponents we need to apply the homomorphism $r$ above.
 They will all have multiplicity $2^{n-2}$ and are from the set
 $$
 \left\lbrace\frac 12 (f_1+f_2),\frac 12 (f_1-f_2), -\frac 12 (f_1-f_2), -\frac 12 (f_1+f_2)\right\rbrace
 $$
 
 \noindent\textbf{Type $D_n$, algebra $\frakg=\so_{2n-2,2}\bR$.}
 We view this orthogonal algebra as the one preserving the indefinite form given by the matrix
 $$
 \begin{bmatrix}
  I_{2n-4} & 0 \\
  0 & J_4
 \end{bmatrix}
 \text{ where }
 J_4=\begin{bmatrix}
    0 & 0 & 0 & 1\\
    0 & 0 & 1 & 0\\
    0 & 1 & 0 & 0\\
    1 & 0 & 0 & 0\\
   \end{bmatrix}
 \text{ and } I_{2n-4} \text{ is an identity matrix}
 $$
 The maximal split subalgebra $\fraka$ is given by matrices of the form
 $$
 \begin{bmatrix}
  0_{2n-4} & 0\\
  0 & T
 \end{bmatrix}
 \text{ where }
 T=
 \begin{bmatrix}
  t_1 & 0 & 0 & 0\\
  0 & t_2 & 0 & 0\\
  0 & 0 & -t_2 & 0\\
  0 & 0 & 0 & -t_1\\
 \end{bmatrix}
 $$
 The weights that occur for the action of $\fraka$ on $\frakg$ are the restricted weights and form a root system of type $B_2$.
 Call the simple roots $\lambda_1,\lambda_2$ (written in a basis $\lambda_1 = f_1-f_2, \lambda_2=f_2$).
 The full root system of $\frakg$ is of type $D_n$, call the simple roots $\alpha_1,\ldots,\alpha_n$ (described in a basis $e_i$ above).
 
 According to \cite[Table 4]{Vinberg}, for $n\geq 4$ the induced morphism $r:\frakh^\dual\to\fraka^\dual$ of root spaces is
 $$
 r(\alpha_i)=\lambda_i \text{ for } i=1,2 \text{ and } r(\alpha_i)=0 \text{ otherwise}
 $$
 For $n=3$ we have the isomorphism $\so_{4,2}\cong \su_{2,2}$.
 The relative root system can be viewed as either $B_2$ or $C_2$, which are isomorphic.
 The morphism $r$ as a map between $D_3$ and $B_2$ becomes
 \begin{align*}
 r(\alpha_1) &= \lambda_1\\
 r(\alpha_2) & = \lambda_2\\
 r(\alpha_3) & = \lambda_2
 \end{align*}
 Note that the morphism $D_3\to B_2$ described above is the same as $A_3\to C_2$ described in \cite[Table 4]{Vinberg} for $\su_{2,2}$.
 
 Solving now the equations for $e_i$, we find independently of $n$ that:
 \begin{align*}
 r(e_1)&=\lambda_1+\lambda_2 = f_1\\
 r(e_2)&=\lambda_2 = f_2\\
 r(e_j)&=0 	\text{ otherwise}
 \end{align*}
 
 Recall the standard representation of $\so$ has highest weight $\varpi_1$ and is denoted $V(\varpi_1)$.
 The weights that occur in it are therefore
 $$
 \Sigma_{V(\varpi_1)} = \{ e_1,\ldots,e_n,-e_n,\ldots,-e_1\}
 $$
 The weights that occur in the restricted root system of $\so_{2n-2,2}$ are the ones which give the Lyapunov spectrum.
 To find them, apply the map $r$ described above to the set of weights to find:
 $$
 \{f_1,f_2,0,\ldots,0,-f_2,-f_1\}
 $$
 
 We also need the spin representations $V(\varpi_{n-1})$ and $V(\varpi_n)$.
 From \cite[Thm. 31.2]{Bump}, each has dimension $2^{n-1}$ and weights
 $$
 \left\lbrace\frac 12 \left(\pm e_1\pm e_2 \pm \cdots\pm e_n\right)\right\rbrace
 $$
 where $V(\varpi_{n-1})$ has an odd number of minus signs, while $V(\varpi_{n})$ an even number.
 
 Again, to find the Lyapunov exponents we need to apply the homomorphism $r$ above.
 They will all have multiplicity $2^{n-3}$ and are from the set
 $$
 \left\lbrace \frac 12 (f_1+f_2),\frac 12 (f_1-f_2), -\frac 12 (f_1-f_2), -\frac 12 (f_1+f_2) \right\rbrace
 $$
\end{example}

\begin{example}
 \label{eg:so*}
 The Lie algebra $\so_{2n}^*$ is a real form of $\so_{2n}(\bC)$, whose real rank is $\lfloor n/2\rfloor$ (the rank of $\so_{2n}(\bC)$ is $n$).
 It can be described as an intersection of a unitary and an orthogonal Lie algebra:
 \begin{align*}
  \so_{2n}^* := \su\left(\bC^{2n}, 
  \begin{bmatrix}
   0 & I_n\\
   I_n & 0
  \end{bmatrix}
  \right)
  \bigcap
  \so\left(\bC^{2n}, 
  \begin{bmatrix}
   I_n & 0\\
   0 & -I_n
  \end{bmatrix}
  \right)
 \end{align*}
 The unitary and orthogonal algebras are required to preserve the hermitian, resp. symmetric forms provided above.
 Their intersection then becomes
 \begin{align*}
  \so_{2n}^* = \left\lbrace
  \begin{bmatrix}
   A & B\\
   B^t & -\conj{A}
  \end{bmatrix}
  \middle\vert 
  A=-A^t, B=-\conj{B}^t
  \right\rbrace
 \end{align*}
 This is viewed as a real Lie algebra, and its standard representation is on $\bC^{2n}=\bR^{4n}$.
 
 To define the maximal real split abelian subalgebra, introduce the $2\times 2$ complex matrix
 $$
 \delta = 
 \begin{bmatrix}
  0 & i\\
  -i & 0
 \end{bmatrix}
 $$
 Note that its eigenvalues when acting on $\bC^2$ are real:
 \begin{align*}
  \delta \cdot
  \begin{bmatrix}
   1\\
   i
  \end{bmatrix}
  &=-1\cdot
  \begin{bmatrix}
   1\\
   i
  \end{bmatrix}\\
  \delta \cdot
  \begin{bmatrix}
   1\\
   -i
  \end{bmatrix}
  &=+1\cdot 
  \begin{bmatrix}
   1\\
   -i
  \end{bmatrix}
 \end{align*}
 A maximal real abelian subalgebra of $\so_{2n}^*$ is then
 \newcommand{\diag}{\operatorname{diag}}
 $$
 \fraka = \left\lbrace
 \begin{bmatrix}
  D & 0\\
  0 & -\conj{D}
 \end{bmatrix} 
 \middle\vert 
 D=\diag\left(t_1\delta,\ldots,t_{\lfloor n/2\rfloor} \delta\right)
 \right\rbrace
 $$
 In the definition, multiples of $\delta$ are placed along the diagonal of $D$ and if $n$ is odd, then the lower-right corner of $D$ has a zero.
 The relative root system of $\so^*_{2n}$ is $C_{\lfloor n/2\rfloor}$ or $BC_{\lfloor n/2\rfloor}$, depending\footnote{See \cite[Table 4]{Vinberg} where $\so_{2n}^*$ is called $\fraku_n^*(\bH)$.} on whether $n$ is even or odd.
 
 \newcommand{\floor}[1]{\lfloor #1\rfloor}
 
 However, to compute the weights in the standard representation the description of $\fraka$ suffices.
 Since the eigenvalues of $\delta$ are $\pm 1$, it follows that the weights of $\fraka$ in the standard representation are (each with multiplicity four)
 $$
 \{ t_1,\ldots t_{\floor{n/2}}, -t_{\floor{n/2}}, \ldots, -t_1 \}
 $$
 If $n$ is odd, there is also a zero with multiplicity four.
\end{example}

\begin{remark}
 The main conclusion from \autoref{eg:orthogonal} is that for spin representations of $\so_{N,2}$ no weights get mapped to zero in the restricted root system.
 From \autoref{eg:so*} it follows that if $n$ is odd and the group is $\so_{2n}^*$, then a zero exponent occurs with multiplicity four. 
 Finally, \autoref{eg:su_pq} shows that for $\su_{p,q}$ the number of zero exponents is as expected, except perhaps in the situation of $\su_{p,1}$.
 This is analyzed in detail in \autoref{prop:su_p1_details}.
\end{remark}

\section{Zero exponents in the Kontsevich-Zorich cocycle}
\label{sec:zero_exponents}

In this section, we put together the results from previous sections and describe all possibilities for zero exponents in the Kontsevich-Zorich cocycle.
This answers affirmatively the first part of the conjecture of Forni, Matheus, and Zorich from \cite{FMZ_zero}.
We begin with explaining a somewhat different mechanism for zero exponents.

\subsection{A cautionary example}
\label{subsec:caut_example}
We present a construction of variations of Hodge structure of weight $1$ in which the number of zero exponents is not predicted by the signature of the pseudo-hermitian form.
The input is a cocycle with monodromy in $\SU_{p,1}$ (carrying a Hodge structure of weight $1$).
Its higher exterior powers still carry a Hodge structure of weight $1$, but the signature of the induced metric does not predict the number of zero exponents.

To fix ideas, we pick concrete numbers below.
The construction is at the level of linear algebra and it then applies to local systems.

\begin{example}
\label{eg:su31}
 Consider $\SU_{3,1}$ acting on $\bC^4=\bC^3\oplus \bC$ in the standard representation, preserving a pseudo-hermitian metric of signature $(3+,1-)$ in the decomposition. 
 Recall that we consider $\bC^n$ as a real vector space of dimension $2n$, equipped with an action of the algebra $\bC$.
 This action commutes with the monodromy (when in $\SU$).
 In particular, the algebra $\bC$ will act on all the Lyapunov spaces.
 
 Consider now the second exterior power (over $\bC$):
 $$
 \wedge^2 _\bC (\bC^3\oplus \bC) = \left(\wedge^2_\bC \bC^3\right) \bigoplus \left(\bC^3\otimes_\bC \bC\right) = \bC^3\oplus \bC^3
 $$
 The signature for this decomposition is $(3+,3-)$.
 In particular, given a cocycle $E$ with monodromy in $\SU_{3,1}$, its second exterior power $\wedge^2_\bC E$ has no predicted zero exponents from the signature alone.
 
 However, the Lyapunov spectrum of $E$ is (written with multiplicities, $\bC$-invariant spaces grouped in parenthesis)
 $$
 (\lambda_1, \lambda_1)\, \, (0,0)\, \, (0,0) \, \,(-\lambda_1, -\lambda_1)
 $$
 Therefore, the spectrum of $\Wedge^2_\bC E$ is
 $$
 (\lambda_1, \lambda_1)\, \, (\lambda_1 ,\lambda_1) \, \,(0,0) \, \,(0,0)\, \, (-\lambda_1, -\lambda_1) \, \,(-\lambda_1, -\lambda_1)
 $$
 Note that the first pair of zeroes comes from $(\lambda_1 \lambda_1)\otimes_\bC (-\lambda_1 -\lambda_1)=(00)$, while the second  from the previous zero exponents.
 
 We now explain why this is compatible with Hodge structures (of weight $1$).
 Recall that giving a Hodge structure on a real vector space $E_\bR$ is the same as giving an action of $\bC^\times$ (viewed as a real algebraic group) on $E_\bR$.
 The $E^{p,q}$ space after complexification corresponds to the the space on which $z\in \bC^\times$ acts as $z^p\conj{z}^q$.
 
 Take now the action of $z\in \bC^\times$ on $\bR^8 = \bC^3\oplus \bC$ to be via $z$ on $\bC^3$ and via $\conj{z}$ on $\bC$.
 The actions should be viewed as by real $8\times 8$ matrices.
 
 After complexification (i.e. taking $(-)\otimes_\bR \bC$) we get eigenspaces of dimensions $4$ each, with action via $z$ and $\conj{z}$ respectively.
 Namely, the earlier $8\times 8$ real matrix diagonalizes, with corresponding complex numbers on the diagonal.

 After taking the second exterior power, we have the decomposition
 $$
 \wedge^2_\bC \left( \bC^3\oplus \bC \right) =\left( \wedge^2_\bC \bC^3\right) \bigoplus \left(\bC^3\otimes_\bC \bC \right)= \bR^{12}
 $$
 The induced action of $z\in \bC^\times$ is by $z^2$ on the first factor, and by $\norm{z}^2$ on the second.
 Twist this action by adding the same scalar action on both factors by $\frac{\conj{z}}{\norm{z}^2}$.
 This gives the desired Hodge structure of weight $1$. 
\end{example}

 The same construction works in general, for $\SU_{p,1}$ and any exterior power.
 The result below gives its general properties.

\begin{proposition}
\label{prop:su_p1_details}
 Consider a variation of Hodge structures $E_\bR$ of weight $1$ over an affine invariant manifold $\cM$.
 Suppose the Zariski closure of the monodromy of $E_\bR$ is $\SU_{p,1}$, or $\operatorname{U}_{p,1}$ (up to finite index), acting in the standard representation on $\bC^{p+1}$.
 
 Consider the $k$-th exterior power $\wedge^k_\bC E_\bR$ of the local system.
 Then it carries an induced variation Hodge structure.
 It becomes of weight $1$ after twisting the circle action, as in \autoref{eg:su31}.
 
 Let now $X$ be the vector field of the \Teichmuller geodesic flow on $\cM$.
 Assume the top exponent of $E_\bR$ is non-zero (equivalently: the monodromy is not contained in a compact unitary group).
 Let also $\sigma$ resp. $\sigma_{\wedge^k}$ be the second fundamental form in direction $X$ of the $(1,0)$ subbundle of $E_\bC$, resp. $(\wedge^k_\bC E_\bR)_\bC$.
 
 Then the following hold:
 \begin{enumerate}
  \item[(i)] The signature of the pseudo-hermitian form on $\wedge^k_\bC E_\bR$ is 
  $$\left(\binom{p}{k}+,\binom{p}{k-1} -\right)$$
  \item[(ii)] The number of zero exponents of $\wedge^k_\bC E_\bR$ is $2\binom{p-1}{k-2}+2\binom{p-1}{k}$.
  \item[(iii)] On a subspace of zero exponents of (real) dimension $2\binom{p-1}{k-2}$ (invariant by the $\bC$-action) the flow preserves a negative-definite metric.
  On a complementary subspace of zero exponents of (real) dimension $2\binom{p-1}{k}$ (also $\bC$-invariant), the cocycle preserves a positive-definite metric.
  \item[(iv)] The number of non-zero exponents of $\wedge^k_\bC E_\bR$ is $2\binom{p-1}{k-1}$.
  \item[(v)] The rank of the second fundamental form $\sigma_{\wedge^k}$ is $\binom{p-1}{k-1}$ at a.e. point in $\cM$.
  In particular it predicts the number of non-zero exponents.
 \end{enumerate}
\end{proposition}
\begin{proof}
% First, we describe the $(1,0)\oplus (0,1)$ splitting of the local system $\wedge^k_\bC E_\bR$.
 Recall that after complexification, we have the splitting
 $$
 E_\bC = E_+\oplus E_-
 $$
 where each $E_\pm$ has a further $(1,0)$ and $(0,1)$ splitting.
 The second fundamental form on the $E_+$ factor is
 $$
 \sigma^+:E^{1,0}_+\to E^{0,1}_+
 $$
 The dimension of the bundles are $p$ and $1$ respectively, so the rank of $\sigma^+$ is $1$.
 
 Let us now describe the structure of $\wedge^k_\bC E_\bR$ after complexification.
 We have
 \begin{align*}
 (\wedge^k_\bC E_\bR)_\bC &= (\wedge^k_\bC E_\bR)_+	&\bigoplus	& (\wedge^k_\bC E_\bR)_-\\
 & = (\wedge^k_\bC E_\bR)_+^{1,0}\oplus (\wedge^k_\bC E_\bR)_+^{0,1}   &\bigoplus &
 (\wedge^k_\bC E_\bR)_-^{1,0}\oplus (\wedge^k_\bC E_\bR)_-^{0,1} 
 \end{align*}
%  The space $(\wedge^k_\bC E_\bR)_+^{1,0}$ is naturally identified with $\wedge^k_\bC (E^{1,0}_+)$.
%  On the second summand however, we have
%  $$
%  (\wedge^k_\bC E_\bR)_-^{1,0} = E^{1,0}_- \otimes \wedge^{k-1}_\bC \left( E_- / E^{1,0}_- \right) \into \wedge^k_\bC \left(E_-\right)
%  $$ 
 The Lyapunov exponents of $E_\bR$ are
 $$
 (\lambda_1, \lambda_1)\, \, (0,0) \cdots (0,0) \, \,(-\lambda_1, -\lambda_1)
 $$
 with $(p-1)$ pairs of zeroes.
 Denote by $E^{\lambda}$ the subspace corresponding to the Lyapunov exponent $\lambda$.
 
 There are $2\binom{p-1}{k-1}$ non-zero exponents in the $k$-th exterior power (over $\bC$).
 Indeed, this subspace corresponds to $E^{\lambda_1}\otimes_\bC \wedge^{k-1}_\bC E^0$.
 
 For zero exponents, there are two types.
 The first kind comes from picking two vectors in the $+\lambda_1$ and $-\lambda_1$ spaces respectively, and then $\binom{p-1}{k-2}$ choices of zero exponent.
 The indefinite metric restricted to this subspace is negative-definite.
 Indeed, the subspace is isomorphic to
 $$
 \wedge^2_\bC(E^{\lambda_1}\oplus E^{-\lambda_1})\otimes_\bC \wedge^{k-2}_\bC E^{0}
 $$
 On the first factor of the tensor product, the metric is negative-definite, on the second it is positive-definite.
 
 The second kind of zero exponent comes from only picking vectors from zero exponent subspaces, giving $\binom{p-1}{k}$ choices.
 It is isomorphic to $\wedge^k_\bC E^{0}$ and the metric is positive-definite on it.
 
 Next we compute the second fundamental form $\sigma_{\wedge^k}^+$.
 On the other summand, it will be the conjugate transpose.
 
 From the Leibnitz rule
 $$
 \nabla_{\wedge^k} (a_1\wedge\cdots\wedge a_k) = \sum_{i} (-1)^{i+1} a_1\wedge\cdots \wedge \nabla a_i \wedge\cdots\wedge a_k
 $$
 we see that
 $$
 \sigma_{\wedge^k}^+(a_1\wedge\cdots\wedge a_k) = \sum_{i} (-1)^{i+1} a_1\wedge\cdots \wedge \sigma^+ a_i \wedge\cdots\wedge a_k
 $$
 We need to evaluate its rank.
 In a fiber we have $\sigma^+:\bC^p \to \bC$, therefore
 $$
 \sigma^+_{\wedge^k}:\wedge^k\bC^p \to \left(\wedge^k(\bC^p\oplus \bC)/\wedge^k\bC^p \right) \cong \bC\otimes \wedge^{k-1}\bC^p
 $$
 Denote by $K\subseteq \bC^p$ the kernel of $\sigma^+$, of dimension $p-1$ at a.e. point in $\cM$.
 Then the kernel of $\sigma^+_{\wedge^k}$ is $\wedge^k K$, and the image is $\bC\otimes \wedge^{k-1}K$.
 Therefore, the rank of $\sigma_{\wedge^k}=\sigma^+_{\wedge^k}\oplus \sigma^-_{\wedge^k}$ is $2\binom{p-1}{k-1}$ and equals the number of non-zero exponents.
\end{proof}

\begin{proposition}
\label{prop:rk_sigma}
 Let $E_\bR$ be a real, irreducible, weight $1$ variation of Hodge structure over an affine manifold.
 Assume the Zariski closure of the monodromy is either $\SU_{p,q}$ in the standard representation, or $\SO^*_{2n}$ in the standard representation.
 
 Consider the Gauss-Manin connection and the second fundamental form of the holomorphic subbundle $F^1$ along the \Teichmuller geodesic flow:
 $$
 \sigma: F^1 \to E_\bC/F^1
 $$
 Then the rank of $\sigma$ is at most $(p+q)-|p-q|$ in the $\SU_{p,q}$ case, or $2n-2$ in the case of $\SO^*_{2n}$ and odd $n$.
 Therefore, the rank is at most the number of strictly positive exponents.
\end{proposition}
\begin{proof}
 When the monodromy is in $\SU_{p,q}$ the complexified bundle splits into two complex-conjugate pieces.
 On one piece $F^1$ has rank $p$ and the quotient rank $q$, while on the other the dimensions are reversed.
 The map $\sigma$ splits accordingly: on one piece it is a map from a rank $p$ bundle to a rank $q$ bundle, and on the other from rank $q$ to rank $p$.
 It follows that $\sigma$ has kernel of dimension at least $|p-q|$.
 
 When the monodromy is in $\SO_{2n}^*$, recall (see \autoref{eg:so*}) that this group was defined as the intersection of (appropriate conjugates of) $\SU_{n,n}$ and $\SO_{2n}(\bC)$.
 Because the monodromy is in $\SU_{n,n}$, the complexified bundle splits as in the previous case.
 
 In addition, however, the bundle $F^1$ is isotropic for an appropriate flat symmetric bilinear form.
 The condition that the symmetric bilinear form is flat for Gauss-Manin implise that $\sigma$ satisfies $\sigma=-\sigma^t$, where transpose is taken for the bilinear form.
 
 Thus at a given point in $\cM$, choosing a basis, $\sigma$ is an $n\times n$ matrix satisfying $\sigma=-\sigma^t$.
 If $n$ is odd, this implies that $\sigma$ has at least $1$-dimensional kernel in each factor of the splitting for $\SU_{n,n}$.
 Indeed, the characteristic polynomial of $\sigma$ will satisfy $p(x)=(-1)^n p(-x)$, and so zero will be a root.
\end{proof}

\subsection{The main result}

We can now summarize the discussion in the previous sections to describe the situations with zero exponents.

\begin{theorem}
\label{thm:KZ_zero}
 For the Kontsevich-Zorich cocycle, the number of zero exponents is precisely equal to the constraints predicted by the monodromy.
 
 Concretely, let $E_\bR$ be one of the flat $\bR$-irreducible pieces of the KZ cocycle.
 Let $G$ be the Zariski closure of the monodromy of $E_\bR$.
 
 Then zero exponents in $E_\bR$ can occur if and only if we are in the following situation.
 The group $G$ has at most one non-compact factor, equal up to finite index to $\SU_{p,q}$, for some $p> q$, or $\SO_{2n}^*$ and $n$ is odd.
 
 The representation in which $\SU_{p,q}$ occurs is the standard one, or an exterior power of the standard.
 In the standard representation, there are $2(p-q)$ zero exponents.
 
 If it is $\SU_{p,q}$ in the $k$-th exterior power of the standard with $k\geq 2$, then necessarily $q=1$.
 The number of zero exponents is then $\binom{p-1}{k-2} + \binom{p-1}{k} $.
 This number is minimum possible, given the monodromy constraint.
 
 If the group is $\SO^*_{2n}$, then zero exponents occur only if $n$ is odd, in which case there are precisely four.
 
 Moreover, the number of strictly positive exponents bounds above the rank of the second fundamental form.
\end{theorem}
\begin{proof}
 By \autoref{prop:real_rk}, the maximal split abelian subalgebra of $G$ acts trivially on a space of dimension at least the number of zero exponents.
 By \autoref{cor:zero_is_rk}, the reverse inequality is also true.
 Namely, \emph{any} cocycle will have at least as many zero exponents as the dimension of the zero restricted weight space.
 
 The classification from \autoref{sec:alg_hull_classification} gives the possibilities that lead to zero exponents.
 Namely, the only possible non-compact factors of $G$ are $\Sp_{2n}$, $\SU_{p,q}$, $\SO_{n,2}$ and $\SO_{2n}^*$.
 
 The analysis done in \autoref{eg:orthogonal} shows that the possible representations of $\SO_{n,2}$ do not have zero weights in the restricted weight spaces decomposition.
 Since $\Sp_{2n}$ can only occur in the standard representation, the same is true for it.
 
 By \autoref{eg:su_pq}, the possible representations of $\SU_{p,q}$ have zero weights in the restricted root space decomposition.
 Moreover, the number of zero exponents is sandwiched between the dimension of the zero restricted weight space. 
 
 The rank of the second fundamental form is bounded above by the number of strictly positive exponents; for exterior power representations of $\SU_{p,1}$ this is \autoref{prop:su_p1_details}(v), for $\SU_{p,q}$ and $\SO^*_{2n}$ this is \autoref{prop:rk_sigma}.
%  By \autoref{rmk:su_pq_index}, the \Teichmuller geodesic flow preserves a norm on the zero Oseledets subspace when the non-compact factors are in $\SU_{p,q}$.
\end{proof}

\subsection{Further remarks}

The second part of the conjecture of Forni, Matheus, and Zorich \cite{FMZ_zero} asks when is the Lyapunov spectrum simple.
Note that if any of the groups $\SO_{n,2}$, $\SO_{2n}^*$, or $\SU_{p,1}$ in an exterior power representation would occur, this would not be the case.
Until recently, all known examples were either of type $\Sp_{2n}$ or $\SU_{p,q}$ in the standard representations.

Recent joint work with Forni and Matheus \cite{FFM_Quaternions} exhibits the first examples of monodromy in the group $\SO^*_{6}$, in the standard representation.
This group coincides with $\SU_{3,1}$ in the second exterior power representation, but the methods developed should in principle lead to monodromy in $\SO^*_{2n}$ for any $n$.

\begin{question}
 Is it possible that some orthogonal group in the spin representation, or $\SU_{p,q}$ in a higher exterior power representation, occur in the algebraic hull of the Kontsevich-Zorich cocycle?
\end{question}

Note that for families of abelian varieties such examples exist.
See for example the constructions of Satake \cite[Section 9]{Satake_examples}.

Given a family of polarized abelian varieties realizing some monodromy group, one can also obtain a family of curves with the same properties.
Indeed, since the family of abelian varieties is polarized, it carries a relatively ample line bundle.
Taking complete interesections gives a generic family of curves on these abelian varieties.
The Jacobians of these curves will then have the abelian varieties as factors.
It seems unlikely that a family arising in \Teichmuller dynamics can be constructed this way.

\paragraph{Rank of the fundamental form.}
Forni, Matheus, and Zorich ask if the rank of the second fundamental form equals the number of strictly positive exponents.
\autoref{thm:KZ_zero} only shows that the number of positive exponents bounds from above the rank of the second fundamental form.

In some situations, this can be used to show that there are no zero exponents.
For example, when Forni's geometric criterion \cite[Thm. 4]{Forni_geometric} applies, using the results of \cite[Sect. 4]{Forni_geometric} one gets lower bounds on the rank of the second fundamental form.
In turn, this gives lower bounds on the number of strictly positive exponents.

% \begin{example}
%  Here we just sketch why the exterior power of the standard representation of $\SU_{p,1}$ admits a Hodge structure of weight $1$.
%  For simplicity, we work with $\operatorname{U}_{p,1}$, in the standard representation on $\bC^{p+1}=\bC^p\oplus \bC$.
%  The natural map of the unit circle $\bS=\bU_1$ to $\operatorname{U}_{p,1}$ is given by requiring that $z=e^{i\theta}$ acts by $z$ on the $\bC^p$ factor, and by $\conj{z}$ on the $\bC$ factor of $\bC^{p+1}$.
%  
%  Then the $k$-th exterior power (over $\bC$) can be decomposed as
%  $$
%  \Lambda^k_\bC \bC^{p+1} = \Lambda^{k}_\bC\bC^p \bigoplus \bC\otimes \Lambda^{k-1}_\bC \bC^p
%  $$
%  The action of $z\in \bU$ on each factor is by $z^k$ and $\conj{z}z^{k-1}=z^{k-2}$ respectively.
%  After twisting the action by $\conj{z}^{k-1}$ on both factors simultaneously, this leads to a Hodge structure of weight $1$.
% \end{example}

%====================================================
%====================================================
%		End of Main Article
%====================================================
%====================================================

%====================================================
%====================================================
%		Appendix
%====================================================
%====================================================

%\appendix
%\section{}

%		End of Appendix
%====================================================

%====================================================
%		Bibliography
%====================================================
% \emergencystretch 1.5em
% \sloppy
\bibliographystyle{sfilip}
\bibliography{zero_exponents}
%====================================================
\end{document}